\documentclass[11pt, a4paper]{article}

\usepackage[latin1]{inputenc}
\usepackage{amsmath, amsfonts, amssymb}
\usepackage[margin = 2.5cm]{geometry}
\usepackage{graphicx, subfig, centernot,eucal}
\usepackage{booktabs}
\usepackage{multirow}

\usepackage{etoolbox}
\makeatletter
\patchcmd{\@maketitle}{\LARGE \@title}{\LARGE\bfseries\@title}{}{}

\renewcommand{\@seccntformat}[1]{\csname the#1\endcsname.\quad}
\makeatother

\usepackage{hyperref}
\hypersetup{
	colorlinks =true,        
	linkcolor =blue,         
	citecolor =blue,         
	urlcolor =blue           
}

\usepackage{amsthm}
\usepackage{thmtools}
\usepackage[capitalise]{cleveref}
\usepackage{xr}

\newtheorem{theorem}{Theorem}[section]
\newtheorem{lemma}[theorem]{Lemma}

\newtheorem{proposition}[theorem]{Proposition}
\newtheorem{problem}{Problem}
\theoremstyle{definition}
\newtheorem{definition}[theorem]{Definition}
\newtheorem{example}[theorem]{Example}
\newtheorem{remark}[theorem]{Remark}

\usepackage[shortlabels]{enumitem}

\allowdisplaybreaks

\newcommand{\I}{\operatorname{I}}
\newcommand{\Fix}{\operatorname{Fix }}

\usepackage[most]{tcolorbox}
\usepackage[normalem]{ulem}
\usepackage[colorinlistoftodos,prependcaption,textsize=tiny]{todonotes}

\begin{document}

\title{Constraint reduction reformulations for projection algorithms with applications to wavelet construction}

\author{Minh N. Dao\thanks{Centre for Informatics and Applied Optimization, School of Engineering, Information Technology and Physical Sciences, Federation University Australia, Ballarat 3353, Australia.
\mbox{E-mail:~\href{mailto:m.dao@federation.edu.au}{m.dao@federation.edu.au}}}
\and
Neil D. Dizon\thanks{School of Mathematical and Physical Sciences, University of Newcastle, Callaghan 2308, Australia.
\mbox{E-mail:~\href{mailto:neil.dizon@newcastle.edu.au}{neilkristofer.dizon@uon.edu.au}}},
\and
Jeffrey A. Hogan\thanks{School of Mathematical and Physical Sciences, University of Newcastle, Callaghan 2308, Australia.
\mbox{E-mail:~\href{jeff.hogan@newcastle.edu.au}{jeff.hogan@newcastle.edu.au}}}, 
\and
Matthew K. Tam\thanks{School of Mathematics and Statistics, The University of Melbourne, Parkville 3010, Australia.
\mbox{E-mail:~\href{matthew.tam@unimelb.edu.au}{matthew.tam@unimelb.edu.au}}}
}

\date{March 16, 2021}

\maketitle

\begin{abstract}
We introduce a reformulation technique that converts a many-set feasibility problem into an equivalent two-set problem. This technique involves reformulating the original feasibility problem by replacing a pair of its constraint sets with their intersection, before applying Pierra's classical product space reformulation. The step of combining the two constraint sets reduces the dimension of the product spaces. We refer to this as the \emph{constraint reduction reformulation} and use it to obtain constraint-reduced variants of well-known projection algorithms such as the Douglas--Rachford algorithm and the method of alternating projections, among others. We prove global convergence of constraint-reduced algorithms in the presence of convexity and local convergence in a nonconvex setting. In order to analyse convergence of the constraint-reduced Douglas--Rachford method, we generalize a classical result which guarantees that the composition of two projectors onto subspaces is a projector onto their intersection. Finally, we apply the constraint-reduced versions of Douglas--Rachford and alternating projections to solve the wavelet feasibility problems, and then compare their performance with their usual product variants.
\end{abstract}


\paragraph{Mathematics Subject Classification (MSC 2020):}
90C26 $\cdot$ 
47H10 $\cdot$ 
65K10 $\cdot$ 
65T60 

\section{Introduction}

A \emph{feasibility problem} is the task of finding a point in the intersection of a finite family of sets. Formally, given sets $K_1, K_2, \dots, K_r$ contained in a Hilbert space, the corresponding feasibility problem is to
\begin{equation}\label{def:feasibility}
	\text{find~} x^\ast \in K :=\bigcap_{j=1}^r K_j. 
\end{equation}
In the literature, \emph{projection algorithms} are often used to solve feasibility problems. The \emph{method of alternating projections} (MAP) \cite{neumann} and the \emph{Douglas-Rachford} (DR) algorithm \cite{drachford} are well-known examples of projection algorithms which are applicable to two-set feasibility problems. Of these, the DR method has experienced sustained popularity because of its empirical potency in nonconvex settings \cite{abtam2,abtam1,aacampoy,bsims,btam1,dtam}. Although originally formulated for two-set feasibility problems, it has been extended to many-set feasibility problems by employing the \emph{cyclic DR method} \cite{btam,btam2}, \emph{cyclically anchored DR} method \cite{bnphan}, \emph{cyclic generalized DR method} \cite{dphan1,dphan2}, or through Pierra's \emph{product space reformulation} \cite{pierra}. The latter reformulation has the potential drawback of computational inefficiency when the number of constraint sets becomes large. This arises because each additional constraint in the original problem results in an additional product-dimension in the reformulation. A scheme to circumvent this is to replace a pair of constraints by their intersection. We formalize this as a new reformulation technique in Section~\ref{sec:mainresults} and use it to introduce variants of well-known projection algorithms. It is further favorable, but not required, if the pair of constraints $K_i$ and $K_j$ satisfy
\begin{equation}\label{def:generalcondition}
	P_{K_i}(K_j) \subseteq K_j
\end{equation}
for some $i,j \in \{1,2,\dots, r\}$ with $i\neq j$, where $P_C$ denotes the $projector$ onto a set $C$. As we will show in Section~\ref{sec:wavelets}, exactly this property appears in the constraint sets arising in the feasibility approach to \emph{wavelet construction} \cite{franklin,fhtam,fhtamfull}.

More precisely, the construction of compactly supported and smooth multidimensional \emph{wavelets} with orthogonal shifts and \emph{multiresolution} structure has been recently formulated as a many-set feasibility problem \cite{franklin,fhtam,fhtamfull} where the DR method, together with other {projection algorithms} and their many-set extensions, has been successfully employed. In this approach, properties of wavelets which are desirable in signal processing (e.g., compact support, smoothness) are treated as constraints alongside the conditions of \emph{multiresolution analysis} (MRA) \cite{mallat,meyer}, and intersection points yield the coefficients of the corresponding scaling and wavelet functions. 

As additional properties such as \emph{real-valuedness}, \emph{symmetry} and \emph{cardinality} \cite{dhlakey} are added to the wavelet construction problem, the computational inefficiencies of the product space reformulation outlined above are realized due to the additional constraint sets. Fortunately, but also rather peculiarly to the structure of the wavelet feasibility problem, its constraint sets satisfy the property stated in \eqref{def:generalcondition}. In particular, we show that the real-valuedness or the symmetry constraint may be combined with constraint sets arising from the conditions of MRA.

The goal of this paper is to present a \emph{constraint reduction reformulation} for projection algorithms aimed at solving the feasibility problem \eqref{def:feasibility}. The main results appear in Section~\ref{sec:mainresults} where we formally introduce the reformulation, and use the framework of \emph{fixed point theory} to study the operators obtained as a result of applying the reformulation to well-known projection algorithms. We give a global convergence analysis for the resulting variant of MAP and DR in the convex setting, and a local convergence analysis in a nonconvex setting. To do so, we extend a classical result regarding commutativity of two projectors on closed subspaces. As we show in Section~\ref{sec:wavelets}, the reformulation can significantly reduce computational time.

The rest of the paper is organized as follows. Section~\ref{sec:preliminaries} recalls relevant preliminaries and auxiliary results. Section~\ref{sec:mainresults} contains the constraint reduction reformulation together with other new results including the generalization of the classical result on the commutativity of two projectors. And finally in Section~\ref{sec:wavelets}, we apply the reformulation to wavelet construction cast as a feasibility problem.

\section{Preliminaries}\label{sec:preliminaries}

Henceforth, we use $\mathcal{H}$ to denote a real Hilbert space endowed with an inner product $\langle \cdot, \cdot \rangle$ and induced norm $\|\cdot\|$. For $x \in \mathcal{H}$ and $\delta \geq 0$, the closed ball centered at $x$ with radius $\delta$ is $\mathbb{B}(x;\delta) := \{z\in \mathcal{H}: \|z-x\| \leq \delta\}$. We use $\I$ to denote the identity mapping on $\mathcal{H}$ which maps any point to itself. Moreover, if $T$ is an operator acting on a set $K$, we write $T(K) = \{T(x): x \in K\}$.  We also denote the \emph{set of fixed points} of the operator $T$ by $\Fix T := \{x \in \mathcal{H}: x \in T(x)\}$, which reduces to $\{x \in \mathcal{H}: x = T(x)\}$ when $T$ is single-valued. Further, the product space $\mathcal{H}^r = \mathcal{H} \times \mathcal{H} \times \dots \times \mathcal{H}$ is also a real Hilbert space endowed with the inner product given by
\begin{equation}\label{def:innerproduct}
	\langle \mathbf{x},\mathbf{y} \rangle = \sum_{j=1}^r \langle x_j,y_j \rangle 
\end{equation}
for all $\mathbf{x} = (x_1,x_2,\dots,x_r)$ and $\mathbf{y} = (y_1,y_2,\dots,y_r)$ in $\mathcal{H}^r$.

\subsection{Projectors, Reflectors and Projection Methods}

\begin{definition}
	Let $C$ be a nonempty subset of $\mathcal{H}$. The \emph{distance function} to $C$ is the function $d_C\colon \mathcal{H}\to \mathbb{R}$ defined by
	\begin{equation*}
		d_C(x) = \inf_{z \in C} \|x-z\|
	\end{equation*}
	and the \emph{projector} onto $C$ is the set-valued operator $P_{C}\colon \mathcal{H} \rightrightarrows C$ defined by
	\begin{equation*}
		P_{C}(x) = \{ c \in C: \|x-c\| = d_C(x) \}.
	\end{equation*}
	The \emph{reflector}  with respect to $C$ is the set-valued operator $R_{C}\colon \mathcal{H} \rightrightarrows \mathcal{H}$ defined by
	\begin{equation*}
		R_{C} := 2P_{C} - \I.
	\end{equation*}
	An element of $P_{C}(x)$ is called a \emph{best approximation} of $x$ from $C$ or a \emph{projection} of $x$ onto $C$. Similarly, an element of $R_C(x)$ is called a \emph{reflection} of $x$ with respect to $C$. If every point in $\mathcal{H}$ has at least one projection onto $C$, then $C$ is said to be \emph{proximinal}.
\end{definition}

Note that the sum in the definition of $R_C$ is understood in the sense of Minkowski set addition. In the case where $P_C$ is single-valued for all $x \in \mathcal{H}$, i.e., $P_C(x)=\{u\}$ for some $u \in C$, we abuse notation by writing $P_C(x)=u$ and understand $P_C$ as a {single-valued operator}. It is a direct consequence of the definition that if the projector onto $C$ is single-valued, then the reflector with respect to $C$ is also single-valued. If the set $C$ is closed and convex, then $P_C$ is single-valued \cite[Theorem~3.5]{deutsch}, and the projections onto $C$ are easily characterized as follows.

\begin{proposition}\label{p:proj}
	Let $C$ be a nonempty closed convex subset of $\mathcal{H}$, and $D$ be a nonempty closed affine subspace of $\mathcal{H}$. Then the following statements hold.
	\begin{enumerate}[label =(\alph*)]
		\item\label{p:proj_convex} 
		$P_C(x) = p$ if and only if $p\in C$ and $\langle x-p,c-p \rangle \leq 0 \text{ for all } c \in C$, 
		\item\label{p:proj_affine} 
		$P_D(x) = p$ if and only if $p\in D$ and $\langle x-p,d-p \rangle = 0 \text{ for all } d \in D$.
	\end{enumerate}
\end{proposition}		
\begin{proof}
	For \ref{p:proj_convex}, see \cite[Theorem~4.1]{deutsch} or \cite[Theorem~3.16]{bcombettes}. For \ref{p:proj_affine}, see \cite[Corollary~3.22]{bcombettes}.
\end{proof}

Projectors and reflectors form part of iterative algorithms called \emph{projection algorithms} for solving feasibility problems. These algorithms exploit the structure of the individual sets which comprise the intersection that is the feasible region. These techniques iterate successively on the individual sets by applying projectors or reflectors, usually in a cyclic fashion.

The earliest formulation of projection methods dates back to the work of von Neumann \cite{neumann} who showed that the sequence $(x_n)_{n\in \mathbb{N}}$ with $x_0 \in \mathcal{H}$ and $x_{n+1} = S_{A,B}(x_n)$, where
\begin{equation*}
	S_{A,B} := P_AP_B,
\end{equation*}
satisfies $\lim_{n \to \infty} x_n = P_{A\cap B}(x_0)$ whenever $A$ and $B$ are closed subspaces. The operator $S_{A,B}$ is sometimes called the \emph{alternating projection operator}, and iterating $S_{A,B}$ to obtain a projection onto the intersection is referred to as the \emph{method of alternating projections}. The result was motivated by von Neumann's return to the question of finding a point on $A\cap B$ when $P_{A}$ and $P_{B}$ do not commute. Before then, it was only known that if $P_{A}$ and $P_{B}$ commute, then $P_{A}P_{B} = P_{A\cap B}$ \cite[Chapter~XIII]{neumann}. The following proposition provides several characterization of this fact.

\begin{theorem}\label{t:commuting}
	If $A$ and $B$ are closed subspaces of $\mathcal{H}$, then the following are equivalent:
	\begin{enumerate}[label =(\alph*)]
		\item $P_{A}P_{B} = P_{B}P_{A}$,
		\item $P_{A}(B) \subseteq B$,
		\item $P_{B}(A) \subseteq A$,
		\item $P_{A}P_{B} = P_{A\cap B}$.
	\end{enumerate}
\end{theorem}
\begin{proof}
	See \cite[Lemma~9.2]{deutsch}.
\end{proof}

In the next section, we generalize Theorem~\ref{t:commuting} to the case where $A$ is a closed affine subspace and $B$ is a proximinal subset. This generalization is key to our analysis of iterative algorithms.

A natural extension of MAP for many-set feasibility problems is the \emph{method of cyclic projections} which iterates by consecutively applying the projectors onto each of the constraint sets. This has a guaranteed convergence when the sets of interest are subspaces \cite{halperin}. Moreover, the method weakly converges to a point on the intersection when the constraint sets are closed and convex \cite{bregman}. For the case of two closed convex sets, if one of the set is compact or if either of the set is finite dimensional with the distance between them being attained, then strong convergence of MAP may be achieved \cite{cheney}.  While there are other projection methods, we confine ourselves mainly to alternating projections, and the DR algorithm which we now introduce.

\begin{definition}
	Given two nonempty subsets $A$ and $B$ of $\mathcal{H}$, the \emph{DR operator} $T_{A,B}$ is defined as
	\begin{equation*}
		T_{A,B} := \frac{\I+R_BR_A}{2}.
	\end{equation*}
\end{definition}
It is worth noting (see \cite[Equations~(20)--(23)]{bdao}) that, if $P_A$ is single-valued, then 
\begin{equation*}
	T_{A,B} = \I - P_A + P_BR_A \text{~~and~~} P_A(\Fix T_{A,B}) = A\cap B.
\end{equation*}
If $A$ and $B$ are closed convex subsets of $\mathcal{H}$ with $A\cap B \neq \varnothing$, then, for any $x_0 \in \mathcal{H}$, the sequence $(x_n)_{n\in \mathbb{N}}$ generated by $x_{n+1} = T_{A,B}(x_n)$ converges weakly to a point $x^\ast \in \Fix T_{A,B}$, and the shadow sequence $(P_A(x_n))_{n\in \mathbb{N}}$ converges weakly to $P_A(x^\ast) \in A\cap B$ \cite{lions,svaiter}.

\subsection{Convergence of Fixed Point Iterations}

Most of the projection algorithms that we have already mentioned can be cast as \emph{fixed point iterations}. That is, for some starting point, a sequence is generated by repeated applications of the operator at hand, ideally to attain a \emph{fixed point} in the limit. In this subsection, we will recall the relevant notions as well as the propositions necessary to establish convergence of projection algorithms.

\begin{definition}\label{def:mappings}
	Let $C\subseteq \mathcal{H}$ and let $T\colon C\to \mathcal{H}$. The mapping $T$ is said to be 
	\begin{enumerate}[label =(\alph*)]
		\item 
		\emph{nonexpansive} if, for all $x,y \in C$,
		\begin{equation*}
			\| T(x)-T(y)\| \leq \|x-y\|;
		\end{equation*}
		\item 
		\emph{firmly nonexpansive} if, for all $x,y \in C$,
		\begin{equation*}
			\| T(x)-T(y)\|^2 + \| (\I-T)(x)-(\I-T)(y)\|^2 \leq \|x-y\|^2;
		\end{equation*} 
		\item 
		\emph{$\alpha$-averaged} if $\alpha \in (0,1)$ and there exists a nonexpansive operator $R\colon C\to \mathcal{H}$ such that 
		\begin{equation*}
			T = (1-\alpha)\I + \alpha R,
		\end{equation*}
		or equivalently, if $\alpha \in (0,1)$ and, for all $x,y \in C$,
		\begin{equation*}
			\| T(x)-T(y)\|^2 + \frac{1-\alpha}{\alpha}\| (\I-T)(x)-(\I-T)(y)\|^2 \leq \|x-y\|^2.
		\end{equation*}
	\end{enumerate}
\end{definition}
It follows from these definitions that $T$ is firmly nonexpansive if and only if it is $1/2$-averaged. Moreover, if $T$ is $\alpha$-averaged, then it is nonexpansive and also $\beta$-averaged with $\beta \in (\alpha, 1)$.

\begin{proposition}\label{ex:projisalpha} 
	Let $C$ be a nonempty closed convex subset of $\mathcal{H}$. Then 
	\begin{enumerate}[label =(\alph*)]
		\item $P_C$ is firmly nonexpansive.
		\item $R_C$ is nonexpansive.
	\end{enumerate}
\end{proposition}
\begin{proof}
	This follows from \cite[Proposition~4.16 and Corollary~4.18]{bcombettes}.
\end{proof} 

It is also easy to establish that the composition of two nonexpansive operators is again nonexpansive. Also, other averaged maps may be obtained from convex combinations and compositions of already known averaged maps.

\begin{proposition}\label{p:averagedoperators}
	Let $C \subseteq \mathcal{H}$ and let $T_j : C \to \mathcal{H}$ be $\alpha_j$-averaged for each $j \in J :=\{1,2,\dots,r\}$. Then the following statements hold.
	\begin{enumerate}[label =(\alph*)]
		\item\label{p:averagedoperators_combination} 
		$\sum_{j\in J}\lambda_j T_j$ is $\alpha$-averaged with $\displaystyle\alpha = \sum_{j\in J}\lambda_j\alpha_j$, whenever $\lambda_j>0$ and $\displaystyle\sum_{j \in J}\lambda_j = 1$;
		\item\label{p:averagedoperators_composition} 
		$T_rT_{r-1}\cdots T_1$ is $\alpha$-averaged with
		\begin{equation*}
			\alpha = 
			\left( 1 + \left(\sum_{j\in J} \frac{\alpha_j}{1-\alpha_j}\right)^{-1} \right)^{-1}.
		\end{equation*}
	\end{enumerate}
\end{proposition}
\begin{proof}
	For \ref{p:averagedoperators_combination}, see \cite[Proposition~4.42]{bcombettes}. For \ref{p:averagedoperators_composition}, see \cite[Proposition~4.46]{bcombettes}.
\end{proof}

The next proposition establishes averagedness as well as characterizes the fixed point set of operators that are coordinate-wise averaged.

\begin{proposition}\label{p:productoperators}
	Let $C_j \subseteq \mathcal{H}$ and let $T_j\colon C_j \to \mathcal{H}$ be $\alpha_j$-averaged for each $j \in J :=\{1,2,\dots,r\}$. Define the operator $T\colon C_1 \times C_2 \times \dots \times C_r\to \mathcal{H}^r$ by
	\begin{equation*}
		T(\mathbf{x}) = (T_1(x_1),T_2(x_2),\dots,T_r(x_r))
	\end{equation*}
	for all $\mathbf{x} = (x_1,x_2,\dots,x_r) \in C_1 \times C_2 \times \dots \times C_r$. Then the following statements hold.
	\begin{enumerate}[label =(\alph*)]
		\item\label{p:productoperators_averaged} 
		$T$ is $\alpha$-averaged with $\alpha = \max\limits_{j \in J} \alpha_j$.
		\item\label{p:productoperators_fix} 
		$\Fix T = \Fix T_1 \times \Fix T_2 \times \cdots \times \Fix T_r$.
	\end{enumerate}
\end{proposition}
\begin{proof}
	\ref{p:productoperators_averaged}: Let $\mathbf{x},\mathbf{y} \in C_1 \times C_2 \times \dots \times C_r$. Since $\alpha = \max_{j\in J} \alpha_j \in (0,1)$, we also have that $T_j$ is $\alpha$-averaged for each $j \in J$. Using Definition~\ref{def:mappings} and \eqref{def:innerproduct}, we obtain
	\begin{align*}
		\|T(\mathbf{x}) -T(\mathbf{y})\|^2 &+ \dfrac{1-\alpha}{\alpha}\|(\I-T)(\mathbf{x}) - (\I-T)(\mathbf{y})\|^2\\
		&=\sum_{j=1}^r\|T_j(x_j)-T_j(y_j)\|^2 + \dfrac{1-\alpha}{\alpha}\sum_{j=1}^r\|(\I-T_j)(x_j) - (\I-T_j)(y_j)\|^2 \\
		&=\sum_{j=1}^r\left(\|T_j(x_j)-T_j(y_j)\|^2 + \dfrac{1-\alpha}{\alpha}\|(\I-T_j)(x_j) - (\I-T_j)(y_j)\|^2 \right)\\
		&\leq \sum_{j=1}^r\|x_j-y_j\|^2\\
		&=\|\mathbf{x}-\mathbf{y}\|^2.
	\end{align*}
	Thus, $T$ is $\alpha$-averaged.
	
	\ref{p:productoperators_fix}: This is immediate from the definition.
\end{proof}

The following proposition provides a useful criterion for convergence of fixed point iterations. 

\begin{proposition}[Opial's theorem]\label{p:Opial}
	Let $T\colon\mathcal{H} \to \mathcal{H}$ be $\alpha$-averaged with $\Fix T \neq \varnothing$. Then, for any $x_0 \in \mathcal{H}$, the sequence $(x_n)_{n\in \mathbb{N}}$ generated by $x_{n+1} = T(x_n)$ converges weakly to a point $x^\ast \in \Fix T$.
\end{proposition}
\begin{proof}
	See \cite{opial}, or set $\lambda_n = 1$ for all $n$ in \cite[Proposition~5.16]{bcombettes}.
\end{proof}

\begin{proposition}\label{p:compositionfixedpoint}
	Let $C$ be a nonempty subset of $\mathcal{H}$ and let $T_j\colon C \to C$ be $\alpha_j$-averaged for each $j \in \{1,2,\dots,r\}$ such that $\bigcap_{j=1}^r \Fix T_j \neq \varnothing$. Then $\Fix (T_r\cdots T_2T_1) = \bigcap_{j=1}^r \Fix T_j$.
\end{proposition}
\begin{proof}
	See \cite[Corollary~4.51]{bcombettes}. 
\end{proof}

\subsection{Product Space Reformulation}\label{sec:productspacetrick}

The \emph{product space reformulation} rewrites a many-set feasibility problem into a two-set feasibility problem \cite{pierra}. Given $K_1, K_2, \dots, K_r \subseteq \mathcal{H}$, with corresponding projectors $P_{K_1}, P_{K_2}, \dots, P_{K_r}$, the sets $C$ and $D$ in the product Hilbert space $\mathcal{H}^r$ are defined by
\begin{subequations}
	\begin{align} 
		C &:= K_1 \times K_2 \times \dots \times K_r \text{ and} \label{productC}\\
		D &:= \{(x_1, x_2, \dots, x_r) \in \mathcal{H}^r: x_1 = x_2 = \cdots = x_r\}. \label{productD}
	\end{align}
\end{subequations}
The $r$-set feasibility problem is equivalent to the two-set feasibility problem on $C$ and $D$ in the sense that 
\begin{equation}\label{productspaceequivalence}
	x^{\ast} \in \bigcap_{j=1}^r K_j  \ \iff \mathbf{x}^{\ast}:=(x^{\ast},x^{\ast},\dots,x^{\ast}) \in C \cap D.
\end{equation}
Furthermore, the projectors onto $C$ and $D$ are given by
\begin{subequations}
	\begin{align}
		P_{C}(\mathbf{x}) &=  P_{K_1}(x_1) \times P_{K_2}(x_2) \times \dots \times P_{K_r}(x_r) \text{ and}  \label{proj:productC}\\\
		P_{D}(\mathbf{x}) &= \left(\frac{1}{r}\sum_{j=1}^r x_j, \frac{1}{r}\sum_{j=1}^r x_j, \dots, \frac{1}{r}\sum_{j=1}^r x_j\right)\label{proj:productD}
	\end{align}
\end{subequations}
for any $\mathbf{x}=(x_1,x_2,\dots, x_r) \in \mathcal{H}^r$; see, e.g., \cite[Proposition~29.3 and Proposition~26.4(iii)]{bcombettes}. Note that $D$ is a closed subspace of $\mathcal{H}$, and $C$ is a closed convex set if and only if $K_1,K_2,\dots,K_r$ are closed and convex.

The product space reformulation allows us to use MAP and DR even when the number of constraint sets is greater than two.

\section{Constraint Reduction for Feasibility Problems}\label{sec:mainresults}

The main objective of this section is to introduce a constraint reduction reformulation for the $r$-set feasibility problem defined in \eqref{def:feasibility}. Before describing the new reformulation, we first prove a generalization of Theorem~\ref{t:commuting}. This will be important in defining a particular case where the resulting operator arising from the constraint reduction reformulation for DR method will have a guaranteed convergence.

\subsection{Projectors onto Intersections}

The following theorem extends Theorem~\ref{t:commuting}, which applies for two closed subspaces of $\mathcal{H}$, to the setting of a closed affine subspace and a proximinal subset.

\begin{theorem}\label{t:main}
	Let $A$ be a closed affine subspace and $B$ be a proximinal subset of $\mathcal{H}$. Consider the following statements.
	\begin{enumerate}[label =(\alph*)]
		\item\label{state1} 
		$P_A(B) \subseteq B$, 
		\item\label{state2} 
		$P_A(B) = A\cap B$,
		\item\label{state3} 
		$P_B(A) \subseteq A$, 
		\item\label{state4} 
		$P_B(A) = A\cap B$, 
		\item\label{state5} 
		$P_BP_A = P_{A\cap B}$.  
	\end{enumerate}
	Then \ref{state1} $\implies$ \ref{state2} $\implies$ \ref{state3} $\implies$ \ref{state4} $\implies$ \ref{state5}. Moreover, if $B$ is convex, then all statements are equivalent.
\end{theorem}
\begin{proof} 
	\ref{state1} $\implies$ \ref{state2}: If $P_A(B)\subseteq B$, then $P_A(B) \subseteq A\cap B	= P_A(A\cap B)\subseteq P_A(B)$, which yields $P_A(B) = A\cap B$.	
	
	\ref{state2} $\implies$ \ref{state3}: Assume that $P_A(B) = A\cap B$. Take any $a\in A$, any $b\in P_B(a)\subseteq B$, and set $p = P_A(b)\in P_A(B) = A\cap B$. Since $A$ is a closed affine subspace, Proposition~\ref{p:proj}\ref{p:proj_affine} gives $\langle a-p, b-p \rangle =0$, which implies that 
	\begin{equation*}
		\|a-b\|^2 = \|a-p\|^2 + \|b-p\|^2.
	\end{equation*}
	As $b\in P_B(a)$ and $p\in B$, it holds that $\|a-b\|\leq \|a-p\|$. Combining with the above equality yields $\|b-p\|^2 = 0$, so $b = p\in A$. Since $a$ was chosen arbitrarily, we deduce that $P_B(A) \subseteq A$. 
	
	\ref{state3} $\implies$ \ref{state4}: This follows by interchanging the roles of $A$ and $B$ in the proof of ``\ref{state1} $\implies$ \ref{state2}''.
	
	\ref{state4} $\implies$ \ref{state5}: Assume that $P_B(A) = A\cap B$. Fix $x\in \mathcal{H}$, let any $b\in P_BP_A(x)$ and any $c\in P_{A\cap B}(x)$. Then $b\in P_B(A) = A\cap B$ and also $c\in A\cap B$. Setting $a = P_A(x)$, we have $b\in P_B(a)$. Since $A$ is a closed affine subspace, Proposition~\ref{p:proj}\ref{p:proj_affine} implies that, for all $z\in B$, $\langle x-a, a-z \rangle =0$, which yields 
	\begin{equation*}
		\|x-b\|^2 = \|x-a\|^2 + \|a-b\|^2 
		\text{~~and~~} 
		\|x-c\|^2 = \|x-a\|^2 + \|a-c\|^2.
	\end{equation*}
	In addition, $\|a-b\|\leq \|a-c\|$ since $b\in P_B(a)$ and $c\in B$. Therefore, $\|x-b\|\leq \|x-c\| = d_{A\cap B}(x)$, which together with $b\in A\cap B$ implies that $\|x-b\| =\|x-c\|$ and $b\in P_{A\cap B}(x)$. From this, we also obtain that $\|a-c\| = \|a-b\| = d_B(a)$, and hence that $c\in P_B(a)$, which means $c\in P_BP_A(x)$. Since $x$, $b$ and $c$ were choosen arbitrarily, we deduce that $P_BP_A = P_{A\cap B}$. 
	
	\ref{state5} $\implies$ \ref{state1}: Assume that $B$ is convex and that $P_BP_A = P_{A\cap B}$. Fix $x \in B$ and set $y:=P_BP_A(x) = P_{A\cap B}(x)$. Then $y \in A \cap B$. Since $A$ is a closed affine subspace and $y \in A$,  Proposition~\ref{p:proj}\ref{p:proj_affine} gives
	\begin{equation}
		\langle P_A(x) -y, P_A(x) - x \rangle = 0.\label{eqn:proof121}
	\end{equation} 
	Moreover, since $x \in B$ and $y=P_BP_A(x)$,  Proposition~\ref{p:proj}\ref{p:proj_convex} yields
	\begin{equation}
		\langle P_A(x)-y,x-y\rangle \leq 0.\label{eqn:proof122} 
	\end{equation}
	Adding \eqref{eqn:proof121} and \eqref{eqn:proof122}, we obtain
	\begin{equation*}
		\langle P_A(x)-y, P_A(x)-y \rangle \leq 0,
	\end{equation*}
	which implies that $P_A(x) = y = P_BP_A(x)$. Thus, $P_A(B)  \subseteq P_B(P_A(B)) \subseteq B$. 
\end{proof}

We remark that if $B$ is not convex, then we need not have $P_A(B) \subseteq B$ when $P_BP_A = P_{A\cap B}$. For a counterexample, we refer to Figure~\ref{fig:EdoesnotimplyA}. Here we take $A = \{(x,x) : x \in \mathbb{R}\}$ and $B=\{(x,y) \in \mathbb{R}^2: 1 \leq x^2+y^2 \leq 4\}$. It is easy to check that $A$ is a subspace, $B$ is nonconvex and $P_BP_A = P_{A\cap B}$. However, $x_0 \in B$ and $P_A(x_0) \notin B$.

\begin{figure}[h!]
	\centering
	\subfloat[$P_BP_A=P_{A\cap B}{\ \ \not\!\!\!\implies} P_A(B)\subseteq B$.\label{fig:EdoesnotimplyA}]{\includegraphics[width = 0.45\linewidth]{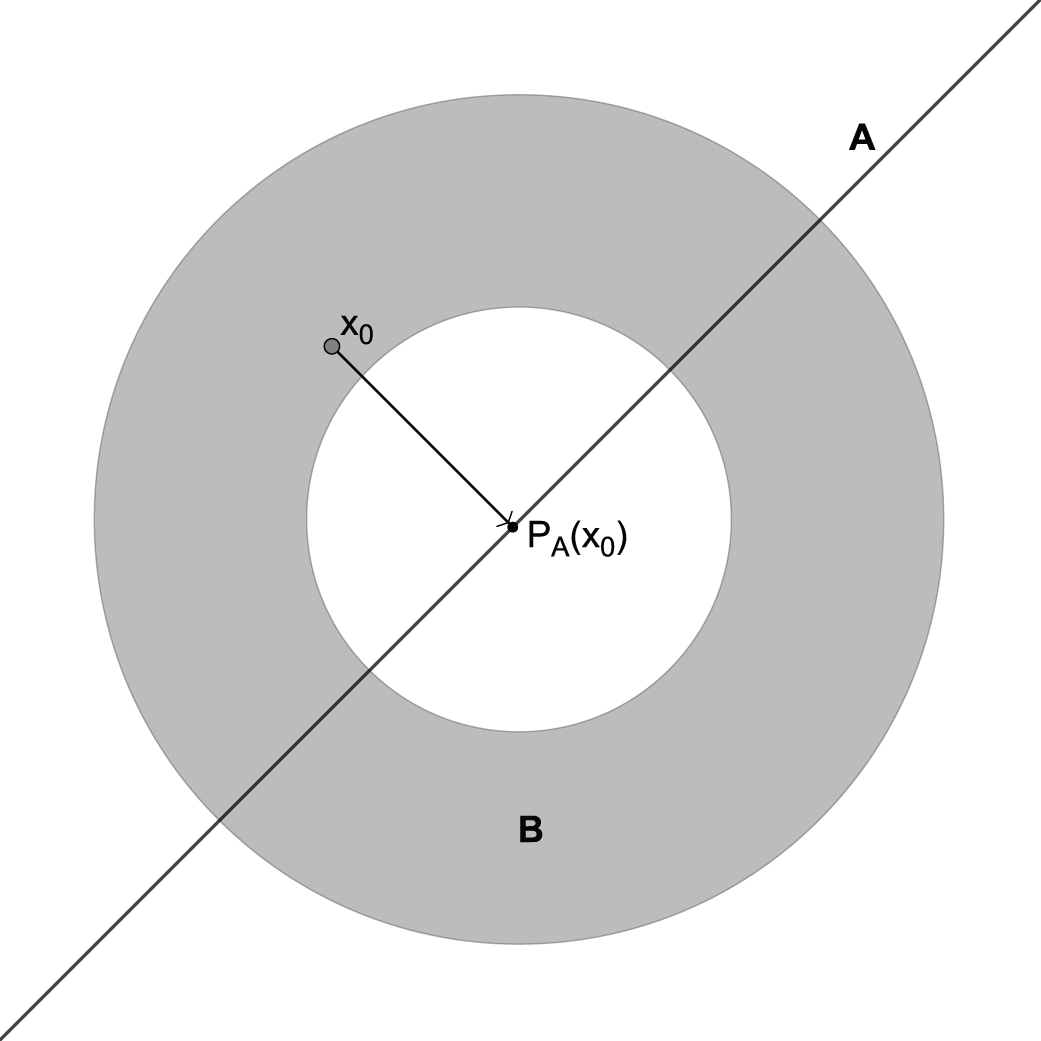}}
	\subfloat[$P_AP_B \neq P_{A\cap B}$.\label{fig:PBPAnot}]{\includegraphics[width = 0.45\linewidth]{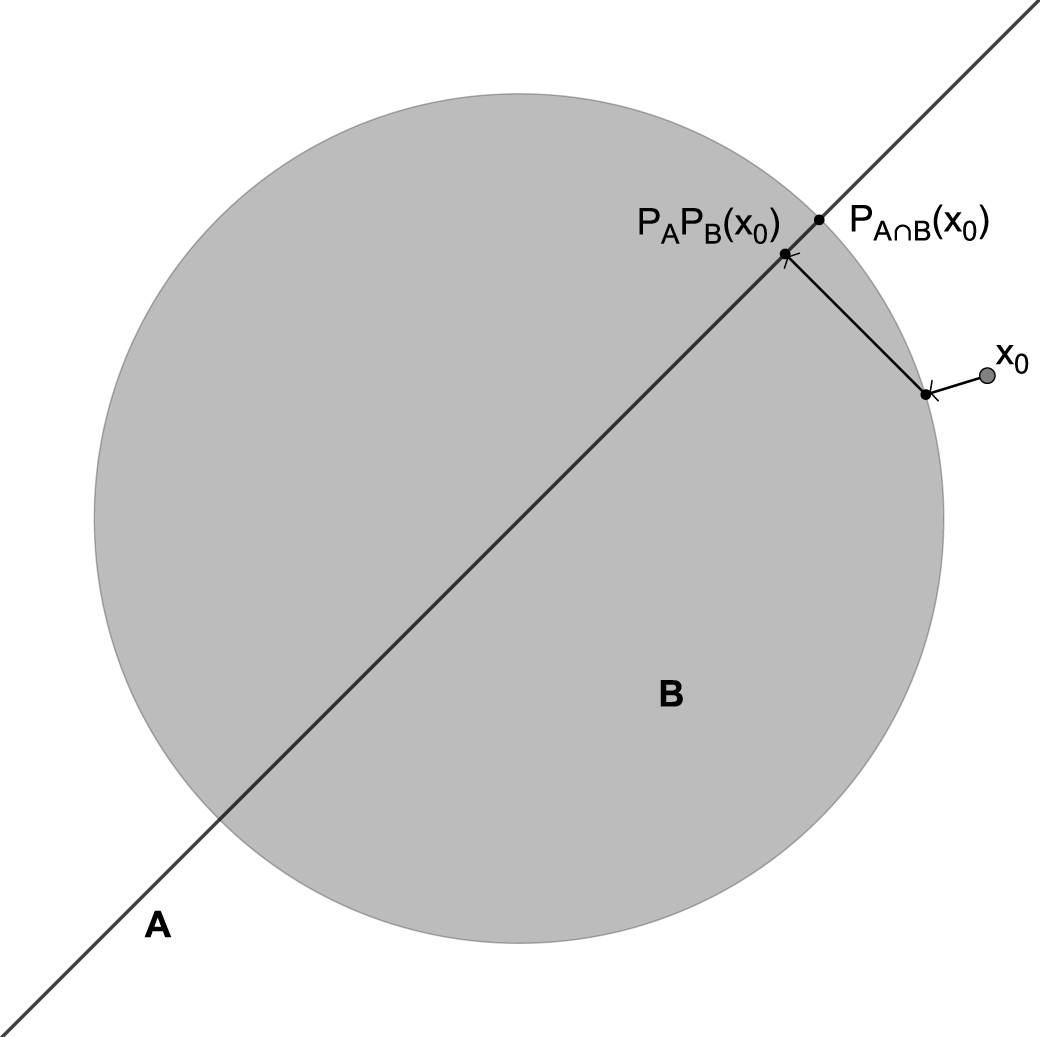}}\\
	\caption{Counterexamples: The figure on the left shows that $P_A(x_0) \notin B$ even if $P_BP_A = P_{A\cap B}$ for a nonconvex set $B$. The figure on the right illustrates that even when $B$ is convex and $P_BP_A = P_{A\cap B}$, it does not follow that $P_A$ and $P_B$ commute.}  \label{fig:myexample}
\end{figure}

Additionally, if $B$ is convex and any of the equivalent statements is true, then it does not follow that $P_AP_B={P_BP_A}$. In particular, $P_AP_B$ need not be equal to $P_{A\cap B}$. To visualize this, refer to Figure~\ref{fig:PBPAnot} where we redefined $B = \{(x,y)\in\mathbb{R}^2 : x^2+y^2 \leq 4\}$ to show that $P_AP_B(x_0) \neq P_{A\cap B}(x_0)$.

\begin{example} 
	We now consider the following examples to illustrate the previous theorem.
	\begin{enumerate}[label =(\alph*)]
		\item 
		Let $A,B \subseteq \mathbb{R}^2$ where 
		\begin{align*}
			A = \{(x,x): x \in \mathbb{R}\},\quad 
			B = \{(x,y): |x| + |y| \leq 1\}.
		\end{align*}
		In this example, $A$ is a subspace and $B$ is a closed convex subset of $\mathbb{R}^2$. Further, $P_BP_A \subseteq A$. Thus, all statements in Theorem~\ref{t:main} hold. In particular, $P_BP_A = P_{A\cap B}$. Refer to Figure~\ref{realconvexexample}. 
		\item 
		Let $A,B \subseteq \mathbb{R}^2$ where 
		\begin{align*}
			A = \{(x,0):x\in \mathbb{R}\},\quad
			B = \left\{(x,y): \sqrt{|x|} + \sqrt{|y|} \leq 1 \right\}.
		\end{align*}
		Here, $A$ is a subspace and $B$ is a nonconvex set. This is a particular example where $B$ is nonconvex but all of the statements in Theorem~\ref{t:main} still hold. Refer to Figure~\ref{realnonconvexexample}.
	\end{enumerate}	
\end{example}

\begin{figure}[h!]
	\centering
	\subfloat[$B$ is convex.\label{realconvexexample}]{\includegraphics[width = 0.5\linewidth]{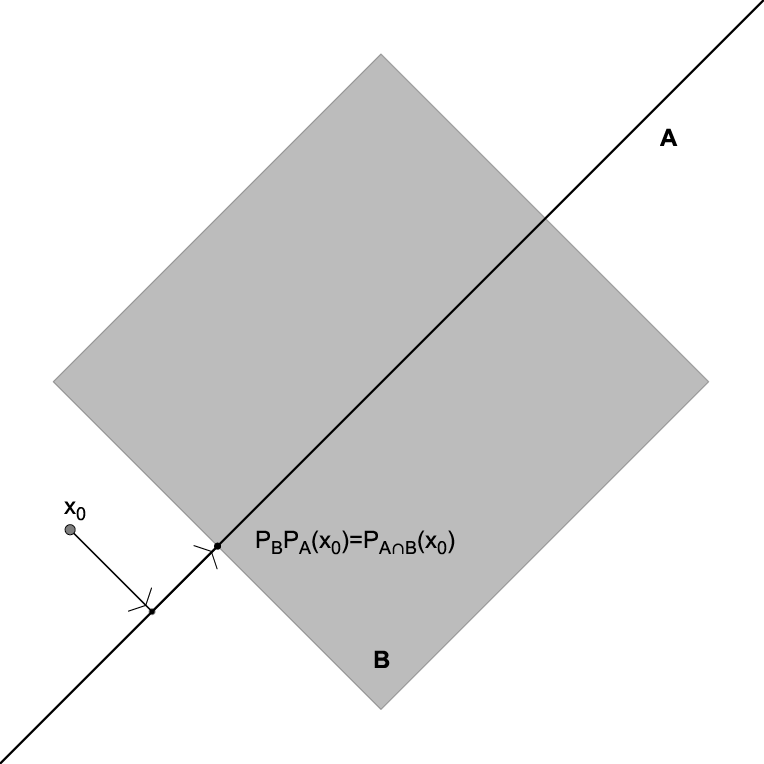}}
	\subfloat[$B$ is nonconvex.\label{realnonconvexexample}]{\includegraphics[width = 0.5\linewidth]{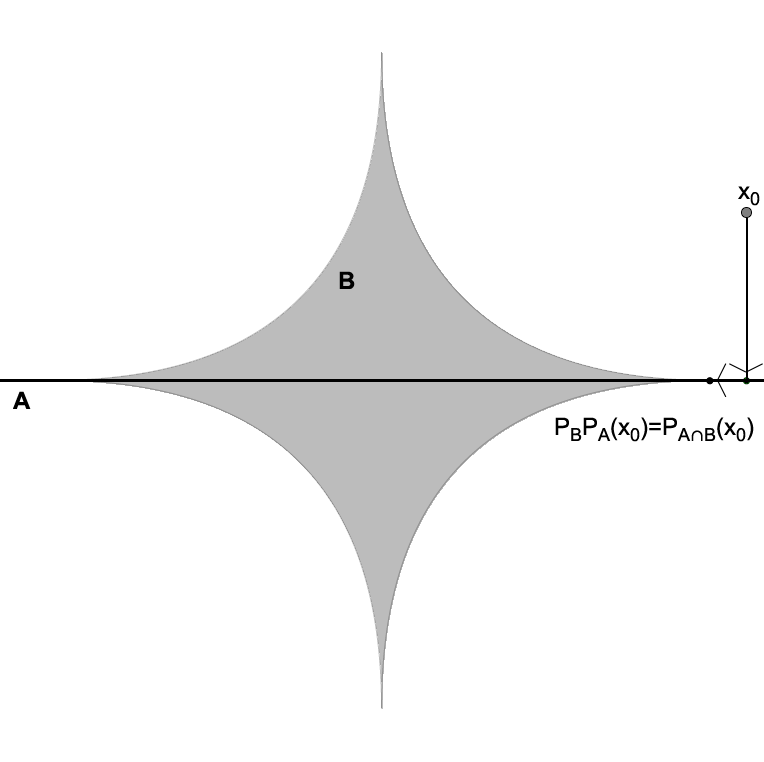}}\\
	\caption{The plot on the left shows that $P_BP_A = P_{A\cap B}$ which is equivalent to all other statements in Theorem~\ref{t:main} because $B$ is convex. The figure on the right is an example satisfying all statements in Theorem~\ref{t:main} even though $B$ is nonconvex.}
\end{figure}

\subsection{Constraint Reduction Reformulation}

The feasibility problem defined in \eqref{def:feasibility} may be solved using a projection algorithm that is applicable to an $r$-set feasibility problem. In particular, one may employ the product DR and the product MAP by defining the product spaces $C$ and $D$ as in \eqref{productC} and \eqref{productD}, respectively.

We now introduce a constraint reduction reformulation which also rewrites an $r$-set feasibility problem into a two-set. This is formalized in the following definition.

\begin{definition}[Constraint reduction reformulation]\label{constraintreduction}
	Let $K_1, K_2, \dots, K_r$ be subsets of $\mathcal{H}$. The \emph{constraint reduction reformulation} of the $r$-set feasibility problem in \eqref{def:feasibility} is the two-set feasibility problem given by
	\begin{equation*}
		\text{find~} \mathbf{x} := (x_1, x_2, \dots, x_{r-1})\in V\cap W\subseteq\mathcal{H}^{r-1},
	\end{equation*}
	where $V$ and $W$ denote the {reduced product space constraints} given by
	\begin{subequations} 
		\begin{align} 
			V &:=K_1\times K_2 \times \cdots \times K_{r-2} \times (K_{r-1}\cap K_r) \subseteq \mathcal{H}^{r-1}, \label{productU}\\
			W &:= \left\{(x_1, x_2, \dots, x_{r-1}) \in \mathcal{H}^{r-1}: x_1 = x_2 = \cdots = x_{r-1}\right\}. \label{productV}
		\end{align}
	\end{subequations}
	The associated mappings $Q_V$ and $P_W$ on $\mathcal{H}^{r-1}$ are defined as
	\begin{subequations}
		\begin{align}
			Q_V(\mathbf{x}) &:= P_{K_1}(x_1)\times P_{K_2}(x_2) \times \dots\times P_{K_{r-2}}(x_{r-2})\times P_{K_{r-1}}P_{K_r}(x_{r-1}), \label{proj:productU}\\
			P_W(\mathbf{x}) &:= \left(\frac{1}{r-1}\sum_{j=1}^{r-1} x_j, \frac{1}{r-1}\sum_{j=1}^{r-1} x_j, \dots, \frac{1}{r-1}\sum_{j=1}^{r-1} x_j\right). \label{proj:productV}
		\end{align}
	\end{subequations}
\end{definition}
This new reformulation can be viewed as two-step process that involves rewriting the original feasibility problem by replacing a pair of its constraint sets with their intersection, followed by an application of Pierra's product space technique to the revised problem with reduced number of constraints. In particular, $K_{r-1} \cap K_{r}$ replaces $K_{r-1}\times K_r$ in the definition of $V$ so that $V$ is a Cartesian product of only $r-1$ sets. The operator $Q_V$ is defined to take the role of $P_C$ by replacing $P_{K_{r-1}}$ and $P_{K_r}$ with the composition $P_{K_{r-1}}P_{K_r}$. Computing $Q_V$ requires the same knowledge about the individual projectors as in $P_C$ in \eqref{proj:productC}. Note however that $Q_V$, in general, is not the projector onto $V$. Furthermore, we note that $W$ is a subspace with dimension one less than that of $D$ defined in \eqref{productD}, and $P_W$ is the projector onto $W$ which takes the role of $P_D$ defined in \eqref{proj:productD}.

We remark that $V$, $W$ and their associated mappings may be reformulated differently to allow for the intersection of other pairs of constraint sets. This will further cut down the dimension of the reduced product space constraints and the ambient Hilbert space. For simplicity of exposition, we focus on the set in Definition~\ref{constraintreduction}, but our results extend to the more general case.

As the following lemma shows, the constraint reduction reformulation still enjoys the equivalence statement \eqref{productspaceequivalence} satisfied by the product space reformulation.

\begin{lemma}\label{l:newproductspaceequivalence}
	Let $K_1, K_2, \dots, K_r$ be subsets of $\mathcal{H}$, and consider $V$ and $W$ as defined in \eqref{productU} and \eqref{productV}, respectively. Then 
	\begin{equation*}
		x^{\ast} \in \bigcap_{j=1}^r K_j  \ \iff \ \mathbf{x}^{\ast}:=(x^{\ast},x^{\ast},\dots,x^{\ast}) \in V \cap W.
	\end{equation*}
\end{lemma}
\begin{proof}
	If $x^{\ast} \in \bigcap_{j=1}^r K_j$, then $x^{\ast} \in K_j$ for all $j \in \{1,2,\dots, r-2\}$ and $x^{\ast} \in K_{r-1}\cap K_r$. Consequently, $\mathbf{x}^{\ast} \in V\cap W$. The reverse implication is straightforward.
\end{proof}

We now apply the constraint reduction reformulation to the method of alternating projections to deduce our first constraint reduced algorithm.

\subsubsection*{Constraint Reduction Reformulation for MAP}

\begin{definition}
	Let $K_1, K_2, \dots, K_r$ be proximinal subsets of $\mathcal{H}$. The \emph{constraint-reduced MAP operator}, denoted by $S$, is defined by
	\begin{equation*}
		S:= P_WQ_V,
	\end{equation*}
	where $Q_V$ and $P_W$ are the operators defined in \eqref{proj:productU} and \eqref{proj:productV}, respectively.
\end{definition}

In the next theorem, we show global convergence of constraint-reduced MAP in the convex setting.

\begin{theorem}\label{t:crMAP}
	Let $K_1, K_2, \dots, K_r$ be closed convex subsets of $\mathcal{H}$ with nonempty intersection. Then the following statements hold.
	\begin{enumerate}[label =(\alph*)]
		\item\label{t:crMAP_fixedpoints} 
		$\Fix S = V\cap W = \{(x_1,x_2,\dots,x_{r-1})\, : \, x_1 = x_2 = \dots = x_{r-1} \in \bigcap_{j=1}^r K_j\}$.
		\item\label{t:crMAP_averaged} 
		$S$ is $3/4$-averaged. If, in addition, $P_{K_{r-1}}(K_r) \subseteq K_r$ and $K_r$ is affine, then $S$ is $2/3$-averaged.
		\item\label{t:crMAP_cvg} 
		For any $\mathbf{x}_0\in \mathcal{H}^{r-1}$, the sequence $(\mathbf{x}_n)_{n\in \mathbb{N}}$ generated by $\mathbf{x}_{n+1} = S(\mathbf{x}_n)$ converges weakly to $\mathbf{x}^\ast = (x^\ast, x^\ast, \dots, x^\ast) \in V \cap W$ with $x^\ast \in \bigcap_{j=1}^r K_j$.
	\end{enumerate}
\end{theorem}
\begin{proof}
	\ref{t:crMAP_fixedpoints}: We first note that $P_{K_j}$ is $1/2$-averaged for each $j \in \{1,2,\dots,r\}$ by Example~\ref{ex:projisalpha}, and then that $P_{K_{r-1}}P_{K_r}$ is $2/3$-averaged by Proposition~\ref{p:averagedoperators}\ref{p:averagedoperators_composition}. We deduce from Proposition~\ref{p:productoperators}\ref{p:productoperators_averaged} that $Q_V$ is $2/3$-averaged. 
	
	Since $\Fix P_{K_j} = K_j$ for each $j \in \{1,2,\dots,r\}$ and $\bigcap_{j=1}^r K_j \neq \varnothing$, we have $\Fix P_{K_{r-1}}P_{K_r} = K_{r-1}\cap K_r$ by Proposition~\ref{p:compositionfixedpoint}, and then $\Fix Q_V = K_1 \times \dots \times K_{r-2} \times (K_{r-1}\cap K_r) = V$ by Proposition~\ref{p:productoperators}\ref{p:productoperators_fix}. Noting that $W$ is also closed convex set, we have from Example~\ref{ex:projisalpha} that $P_W$ is $1/2$-averaged. Moreover, $\Fix P_W = W$ and, by Lemma~\ref{l:newproductspaceequivalence}, 
	\begin{equation*}
		V\cap W = \left\{(x_1,x_2,\dots,x_{r-1})\, : \, x_1 = x_2 = \dots = x_{r-1} \in \bigcap_{j=1}^r K_j\right\} \neq \varnothing.
	\end{equation*}
	Applying Proposition~\ref{p:compositionfixedpoint} again to $Q_V$ and $P_W$ gives us $\Fix S = \Fix Q_V\cap \Fix P_W = V\cap W$. 
	
	\ref{t:crMAP_averaged}: As shown in \ref{t:crMAP_fixedpoints}, $Q_V$ is $2/3$-averaged and $P_W$ is $1/2$-averaged. By applying Proposition~\ref{p:averagedoperators}\ref{p:averagedoperators_composition}, $S$ is $3/4$-averaged. Let us now assume that $P_{K_{r-1}}(K_r) \subseteq K_r$ and $K_r$ is affine. Theorem~\ref{t:main} yields $P_{K_{r-1}}P_{K_r} = P_{K_{r-1}\cap K_r}$. This makes $P_{K_{r-1}}P_{K_r}$ and $Q_V$ both $1/2$-averaged. Consequently, $S$ is $2/3$-averaged as given again by Proposition~\ref{p:averagedoperators}\ref{p:averagedoperators_composition}.  
	
	\ref{t:crMAP_cvg}: We have from \ref{t:crMAP_fixedpoints} that $\Fix S = V\cap W \neq \varnothing$. Since $S$ is $3/4$-averaged by \ref{t:crMAP_averaged}, invoking Proposition~\ref{p:Opial} yields the desired result.
\end{proof}

\begin{remark}
	Without the additional assumptions of Theorem~\ref{t:crMAP}\ref{t:crMAP_averaged}, the operator $P_{K_{r-1}}P_{K_r}$ is not $1/2$-averaged in general, even when $K_{r-1}$ and $K_r$ are both closed subspaces of $\mathbb{R}^2$ \cite[Example~4.2.5]{bauschke1997method}. As a consequence, the operator $Q_V$ is not $1/2$-averaged in general. Nevertheless, these extra assumptions are not necessary in obtaining the fixed point result in  Theorem~\ref{t:crMAP}\ref{t:crMAP_fixedpoints} and the convergence result described in Theorem~\ref{t:crMAP}\ref{t:crMAP_cvg}. When these assumptions are present, Theorem~\ref{t:crMAP}\ref{t:crMAP_cvg} follows from Theorem~\ref{t:main}\ref{state1}\&\ref{state5} and the convergence analysis of MAP for two closed convex sets.
\end{remark}

\subsubsection*{Constraint Reduction Reformulation for DR}

\begin{definition}
	Let $K_1, K_2, \dots, K_r$ be proximinal subsets of $\mathcal{H}$. The \emph{constraint-reduced DR operator}, denoted by $T$, is defined by
	\begin{equation*}
		T := \I - P_W + Q_VR_W = \I - P_W + Q_V(2P_W-\I), 
	\end{equation*}
	where $Q_V$ and $P_W$ are the operators defined in \eqref{proj:productU} and \eqref{proj:productV}, respectively.
\end{definition}

We reiterate that $Q_V$ is not necessarily a projector onto $V$, so that the classic convergence results for DR (or product DR) do not easily follow for $T$. Although a similar characterization of its fixed points still holds, we do not have a general convergence result analogous to Theorem~\ref{t:crMAP} for the constraint-reduced DR. But in particular cases where we know more about the structure of $K_{r-1}$ and $K_r$, we can prove convergence.

\begin{theorem}\label{t:crDR}
	Let $K_1, K_2, \dots, K_r$ be proximinal subsets of $\mathcal{H}$ with nonempty intersection. Suppose that $P_{K_{r-1}}(K_r) \subseteq K_r$. Then the following statements hold.
	\begin{enumerate}[label =(\alph*)]
		\item\label{t:crDR_fixedpoints} 
		$P_W(\Fix T) = V\cap W$. In particular, if $(x_1,x_2,\dots,x_{r-1}) \in \Fix T$, then
		\begin{equation*}
			\frac{1}{r-1}\sum_{j=1}^{r-1} x_j \in \bigcap_{j=1}^r K_j.
		\end{equation*}
		
		\item\label{t:crDR_DR}
		If $K_r$ is affine, then $T = (\I + R_VR_W)/2$ coincides with the DR operator for $W$ and $V$. 
		\item\label{t:crDR_cvg} 
		If $K_1, \dots K_{r-1}$ are convex and $K_r$ is affine, then $T$ is firmly nonexpansive. Consequently, for any $\mathbf{x}_0\in \mathcal{H}^{r-1}$, the sequence $(\mathbf{x}_n)_{n\in \mathbb{N}}$ generated by $\mathbf{x}_{n+1} = T(\mathbf{x}_n)$ converges weakly to a point $\mathbf{x}^\ast = (x_1^\ast, x_2^\ast, \dots, x_{r-1}^\ast) \in \Fix T$. Moreover, writing $\mathbf{x}_n = (x_{1,n}, x_{2,n}, \dots, x_{r-1,n})$, the sequence $\left(\frac{1}{r-1}\sum_{j=1}^{r-1} x_{j,n}\right)_{n\in \mathbb{N}}$ converges weakly to 
		\begin{equation*}
			\frac{1}{r-1}\sum_{i=j}^{r-1} x_j^\ast \in \bigcap_{j=1}^r K_j.
		\end{equation*}
		
	\end{enumerate}
\end{theorem}
\begin{proof} 
	\ref{t:crDR_fixedpoints}: First, it follows from $P_{K_{r-1}}(K_r) \subseteq K_r$ that $P_{K_{r-1}}P_{K_r}(\mathcal{H}) = P_{K_{r-1}}(K_r) \subseteq K_{r-1}\cap K_r$, and thus, 
	\begin{equation*}
		Q_V(\mathcal{H}^{r-1}) \subseteq K_1 \times \dots \times K_{r-2} \times K_{r-1}\cap K_r = V.
	\end{equation*}
	Let $\mathbf{x} \in \Fix T$. Then $\mathbf{x} \in T(\mathbf{x}) = \mathbf{x} - P_W(\mathbf{x}) + Q_VR_W(\mathbf{x})$, which implies that $P_W(\mathbf{x}) \in Q_VR_W(\mathbf{x}) \subseteq Q_V(\mathcal{H}^{r-1}) \subseteq V$. Therefore, $P_W(\mathbf{x}) \in V\cap W$. We deduce that $P_W(\Fix T) \subseteq V\cap W$. On the other hand, it is straightforward to see that $V\cap W \subseteq \Fix T$, which yields $V\cap W = P_W(V\cap W) \subseteq P_W(\Fix T)$. Hence, $P_W(\Fix T) = V\cap W$. 
	
	Now, if $\mathbf{x} = (x_1,x_2,\dots,x_{r-1}) \in \Fix T$, then 
	\begin{equation*}
		P_W(\mathbf{x}) = \left(\frac{1}{r-1}\sum_{j=1}^{r-1} x_j, \frac{1}{r-1}\sum_{j=1}^{r-1} x_j, \dots, \frac{1}{r-1}\sum_{j=1}^{r-1} x_j\right) \in P_W(\Fix T) = V\cap W,
	\end{equation*}
	and the conclusion follows from Lemma~\ref{l:newproductspaceequivalence}.
	
	\ref{t:crDR_DR}: Assume that $K_r$ is affine. Since $K_r$ is proximinal, it is closed (see \cite[Theorem~3.1]{deutsch}), and we have that $K_r$ is a closed affine subspace. Using Theorem~\ref{t:main}, $P_{K_{r-1}}P_{K_r} = P_{K_{r-1}\cap K_r}$, and so $Q_V = P_V$ is the projector onto $V$. This implies that $T = \I - P_W + P_VR_W = (\I+R_VR_W)/2$ is the DR operator for $W$ and $V$. 
	
	\ref{t:crDR_cvg}: Assume that $K_1, \dots K_{r-1}$ are convex and $K_r$ is affine. By \ref{t:crDR_DR}, $T = (\I+R_VR_W)/2$. Since every proximinal set in a Hilbert space is closed (see \cite[Theorem~3.1]{deutsch}), we derive that $V$ is convex and closed. As $W$ is also convex and closed, by Example~\ref{ex:projisalpha}, $R_WR_V$ is nonexpansive and hence $T$ is $1/2$-averaged, i.e., firmly nonexpansive. 
	
	Finally, since $\bigcap_{j=1}^r K_j \neq \varnothing$, Lemma~\ref{l:newproductspaceequivalence} implies that $V\cap W\neq \varnothing$. The weak convergence of $(\mathbf{x}_n)_{n\in \mathbb{N}}$ to $\mathbf{x}^\ast \in \Fix T$ follows from Proposition~\ref{p:Opial}, see also \cite[Theorem~1]{lions}. We also derive from \cite[Theorem~1]{svaiter} that $(P_W(\mathbf{x}_n))_{n\in \mathbb{N}}$ converges weakly to $P_W(\mathbf{x}^\ast) \in P_W(\Fix T) = V\cap W$. This completes the proof. 
\end{proof}

We wish to highlight that the constraint reduction reformulation for closed convex sets $K_1,K_2,\dots,K_r$ with  additional assumptions that $P_{K_{r-1}}(K_r) \subseteq K_r$ and that $K_r$ is a closed affine subspace, coincides with a non-standard application of the product space reformulation since $P_{K_{r-1}}P_{K_r} = P_{K_{r-1}\cap K_r}$ by Theorem~\ref{t:main}. On the other hand, if we lift the convexity assumption on at least the set $K_{r-1}$ but assuming it is proximinal, then we still have $P_{K_{r-1}}P_{K_r} = P_{K_{r-1}\cap K_r}$ by Theorem~\ref{t:main}\ref{state1}\&\ref{state5}. This makes $Q_V$ a projector so that the convergence results can be deduced from the convergence analysis of the DR algorithm for two closed convex sets. In this case, the projector is no longer guaranteed to be nonexpansive and thus the convergence results for constraint-reduced operators like $S$ or $T$ do not necessarily follow. As we will see in  the next section, local convergence in nonconvex settings can still be guaranteed by replacing convexity with set regularity notions.

We end this section by noting that the one dimension reduction in the product spaces $V$ and $W$ is consequential to combining the pair of constraint sets $K_{r-1}$ and $K_r$. In general, given an $r$-set feasibility problem, we may pair up as many sets as possible, and replace each pair by their intersection to form the reformulated problem. This will allow for more reduction in dimensionality. It is relatively easy to read off from the proof of Theorem~\ref{t:crMAP} that such a problem reformulation will still yield a similar fixed point and global convergence results for the corresponding constraint-reduced MAP. Similarly, a corresponding constraint-reduced DR may be set up for solving such a reformulated problem. However, for a favorable fixed point result, the proof of Theorem~\ref{t:crDR} suggests that we must be clever in pairing up any two sets $K_i$ and $K_j$ in that they must satisfy $P_{K_i}(K_j) \subseteq K_j$, for $i,j \in \{1,2,\ldots,r\}$ with $i \neq j$. Moreover, for convergence, $K_j$ must be affine.

\subsection{Local Convergence of Constraint Reduced Algorithms}

In this subsection, $\mathcal{H}$ is finite-dimensional. Then a nonempty set in $\mathcal{H}$ is proximinal if and only if it is closed; see \cite[Corollary~3.15]{bcombettes}. Let $C$ be a nonempty closed subset of $\mathcal{H}$. The \emph{limiting normal cone} to $C$ at $x\in C$ (see \cite[Definition~1.1(ii) and Theorem~1.6]{mordukhovich}) can be given by 
\begin{equation*}
	N_C(x) = \left\{ \lim_{n\to +\infty} \lambda_n(z_n-x_n)\, : \, \lambda_n\geq 0, x_n\to x \text{~with~} z_n\in P_C(x_n) \right\}.
\end{equation*}
Recall from \cite{llmalick} that $C$ is \emph{superregular} at a point $x\in C$ if, for any $\varepsilon >0$, there exists $\delta >0$ such that, for all $y, z\in C\cap \mathbb{B}(x;\delta)$ and all $u\in N_C(z)$,
\begin{equation*}
	\langle u, y-z \rangle \leq \varepsilon\|u\|\|y-z\|.
\end{equation*}
A family of sets $\{K_1, K_2, \dots, K_r\}$ in $\mathcal{H}$ is said to be 
\begin{enumerate}[label =(\alph*)]
	\item 
	\emph{linearly regular} around $x\in \mathcal{H}$ if there exist $\kappa\geq 0$ and $\delta >0$ such that, for all $z\in \mathbb{B}(x;\delta)$,
	\begin{equation*}
		d_{K_1\cap K_2\cap \dots \cap K_r}(z) \leq \kappa\max\{d_{K_1}(z), d_{K_2}(z), \dots, d_{K_r}(z)\}.
	\end{equation*}    
	\item 
	\emph{strongly regular} at $x\in \mathcal{H}$ if 
	\begin{equation*}
		u_1 + u_2 + \dots + u_r = 0 \text{~~with~~} u_j\in N_{K_j}(x) \implies u_1 =u_2 =\dots =u_r = 0.    
	\end{equation*}
	When $r=2$, the strong regularity condition can be written as 
	\begin{equation*}
		N_{K_1}(x)\cap (-N_{K_2}(x)) = \{0\}.
	\end{equation*}
\end{enumerate}
Interested readers can find more discussion on linear regularity and strong regularity in \cite{bbauschke,dphan1,dphan2,kruger,llmalick,russell2018quantitative}.

\begin{proposition}\label{p:superregular}
	Let $\{K_1, K_2, \dots, K_r\}$ be a family of sets in $\mathcal{H}$. The following statements hold. 
	\begin{enumerate}[label =(\alph*)]
		\item\label{p:superregular_product} 
		If $K_j$ is superregular at $x_j\in K_j$ for each $j\in \{1,\dots,r\}$, then the product set $C:=K_1\times K_2\times\dots\times K_r\subseteq\mathcal{H}^r$ is superregular at $\mathbf{x}:=(x_1,x_2,\dots,x_r)\in C$. 
		\item\label{p:superregular_intersection} 
		If $K_j$ is superregular at $x\in K :=\bigcap_{j=1}^r K_j$ for each $j\in \{1,\dots,r\}$ and $\{K_1, K_2, \dots, K_r\}$ is strongly regular at every $z$ near $x$, then the intersection set $K$ is superregular at $x$.
	\end{enumerate}
\end{proposition}
\begin{proof}
	Let $\varepsilon>0$. 
	
	\ref{p:superregular_product}: Since $K_j$ is superregular at $x_j$, there exists $\delta_j>0$ such that, for all $y_j,z_j\in K_j\cap \mathbb{B}(x_j;\delta_j)$ and all $u_j\in N_{K_j}(z_j)$, we have
	\begin{equation}\label{eq:superreg j}
		\langle u_j,y_j-z_j\rangle \leq  \varepsilon\|u_j\|\|y_j-z_j\|. 
	\end{equation}
	Set $\delta=\min_{j=1,\dots,r}\delta_j$. Let $\mathbf{y}=(y_1,\dots,y_r),\mathbf{z}=(z_1,\dots,z_r)\in C\cap\mathbb{B}(\mathbf{x};\delta)$ and $\mathbf{u}=(u_1,\dots,u_n)\in N_C(z)=N_{K_1}(z)\times\dots\times N_{K_r}(z)$. Then \eqref{eq:superreg j} followed by the Cauchy--Schwarz inequality yields
	\begin{align*}
		\langle \mathbf{u},\mathbf{y}-\mathbf{z}\rangle
		= \sum_{j=1}^r\langle u_j,y_j-z_j\rangle 
		&\leq \varepsilon \sum_{j=1}^r\|u_j\|\|y_j-z_j\| \\
		&\leq \varepsilon \left( \sum_{j=1}^r\|u_j\|^2\right)^{1/2}\left( \sum_{j=1}^r\|y_j-z_j\|^2\right)^{1/2} \\
		&= \varepsilon\|\mathbf{u}\|\|\mathbf{y}-\mathbf{z}\|,
	\end{align*}
	which establishes the result.
	
	\ref{p:superregular_intersection}: Since $K_j$ is superregular at $x$, there exists $\delta >0$ such that, for all $y, z\in K\cap \mathbb{B}(x;\delta)\subseteq K_j\cap \mathbb{B}(x;\delta)$ and all $u_j\in N_{K_j}(z)$, we have
	\begin{equation*}
		\langle u_j,y-z \rangle \leq \varepsilon\|u_j\|\|y-z\|. 
	\end{equation*}
	Let $u\in N_{K}(z)$ be arbitrary. By assumption, shrinking $\delta$ if necessary, $\{K_1, K_2, \dots, K_r\}$ is strongly regular at $z$ and, by \cite[Corollary~3.37]{mordukhovich}, $u = \sum_{j=1}^r u_j$ with some $u_j\in N_{K_j}(z)$. We then derive from \cite[Proposition~2.4]{dphan2} and \cite[Theorem~1.6]{mordukhovich} the existence of $\zeta >0$ independent of $u_j$'s and $u$ such that 
	\begin{equation*}
		\|u\| = \Big\|\sum_{j=1}^r u_j\Big\| \geq \zeta\sum_{j=1}^r \|u_j\|.
	\end{equation*}
	Therefore,
	\begin{equation*}
		\langle u,y-z \rangle = \sum_{j=1}^r \langle u_j,y-z \rangle \leq \varepsilon\sum_{j=1}^r \|u_j\|\|y-z\| \leq \frac{\varepsilon}{\zeta}\|u\|\|y-z\|,
	\end{equation*}
	which completes the proof.
\end{proof}

\begin{proposition}\label{p:equi}
	Let $\{K_1, K_2, \dots, K_r\}$ be a family of sets in $\mathcal{H}$ and set
	\begin{align*} 
		C &:= K_1 \times K_2 \times \dots \times K_r \text{~and} \\
		D &:= \{(x_1, x_2, \dots, x_r) \in \mathcal{H}^r: x_1 = x_2 = \cdots = x_r\}.
	\end{align*}
	Then the following statements hold.
	\begin{enumerate}[label =(\alph*)]
		\item\label{p:equi_linreg} 
		$\{K_1, K_2, \dots, K_r\}$ is linearly regular around $x\in \mathcal{H}$ if and only if $\{C,D\}$ is linearly regular around $(x, x, \dots, x)\in \mathcal{H}^r$.
		\item\label{p:equi_strreg} 
		$\{K_1, K_2, \dots, K_r\}$ is strongly regular at $x\in \mathcal{H}$ if and only if $\{C,D\}$ is strongly regular at $(x, x, \dots, x)\in \mathcal{H}^r$.
	\end{enumerate}
\end{proposition}
\begin{proof}
	\ref{p:equi_linreg}: Set $K :=K_1\cap K_2\cap \dots \cap K_r$. We first have that, for all $\mathbf{z} =(z, z, \dots, z)\in D$, 
	\begin{equation}\label{e:dCD}
		d_{C\cap D}^2(\mathbf{z}) = \inf_{\mathbf{y} = (y, y, \dots, y)\in C\cap D} \|\mathbf{z}-\mathbf{y}\|^2 = r\inf_{y\in K} \|z-y\|^2 = rd_{K}^2(z)
	\end{equation}
	and, since $P_C(\mathbf{z}) =P_{K_1}(z) \times P_{K_2}(z) \times \dots \times P_{K_r}(z)$,
	\begin{equation}\label{e:dC}
		d_C^2(\mathbf{z}) = d_{K_1}^2(z) + d_{K_2}^2(z) + \dots + d_{K_r}^2(z).
	\end{equation}
	
	Assume that $\{K_1, K_2, \dots, K_r\}$ is linearly regular around $x\in \mathcal{H}$. Then there exist $\kappa\geq 0$ and $\delta >0$ such that, for all $z\in \mathbb{B}(x;\delta)$, we have
	\begin{equation}\label{e:dK}
		d_{K}(z) \leq \kappa\max\{d_{K_1}(z), d_{K_2}(z), \dots, d_{K_r}(z)\}.
	\end{equation} 
	Set $\mathbf{x} :=(x, x, \dots, x)\in \mathcal{H}^r$ and let $\mathbf{y}\in \mathbb{B}(\mathbf{x};\sqrt{r}\delta/2)$ and $\mathbf{z} =P_D(\mathbf{y})$. Noting that $\mathbf{x}\in D$, we have
	\begin{equation*}
		\|\mathbf{z}-\mathbf{x}\| \leq \|\mathbf{y}-\mathbf{z}\| + \|\mathbf{y}-\mathbf{x}\| \leq 2\|\mathbf{y}-\mathbf{x}\| \leq \sqrt{r}\delta.
	\end{equation*}
	Thus, $\mathbf{z} =(z, z, \dots, z)\in D$ with $z\in \mathbb{B}(x;\delta)$. It follows from \eqref{e:dCD}, \eqref{e:dC}, and \eqref{e:dK} that 
	\begin{equation*}
		d_{C\cap D}^2(\mathbf{z}) = rd_K^2(z) \leq r\kappa^2\max\{d_{K_1}^2(z), d_{K_2}^2(z), \dots, d_{K_r}^2(z)\} \leq r\kappa^2d_C^2(\mathbf{z}),
	\end{equation*}
	and so 
	\begin{equation*}
		d_{C\cap D}(\mathbf{z}) \leq \sqrt{r}\kappa d_C(\mathbf{z}) \leq \sqrt{r}\kappa(d_C(\mathbf{y}) + \|\mathbf{y}-\mathbf{z}\|) = \sqrt{r}\kappa(d_C(\mathbf{y}) + d_D(\mathbf{y})).
	\end{equation*}
	We deduce that 
	\begin{align*}
		d_{C\cap D}(\mathbf{y}) &\leq d_{C\cap D}(\mathbf{z}) + \|\mathbf{y}-\mathbf{z}\| = d_{C\cap D}(\mathbf{z}) + d_D(\mathbf{y}) \\
		&\leq \sqrt{r}\kappa d_C(\mathbf{y}) + (1+\sqrt{r}\kappa)d_D(\mathbf{y}) \\
		&\leq (1+2\sqrt{r}\kappa)\max\{d_C(\mathbf{y}), d_D(\mathbf{y})\},
	\end{align*}
	which implies the linear regularity of $\{C,D\}$ around $\mathbf{x}$.
	
	Conversely, assume that $\{C,D\}$ is linearly regular around $\mathbf{x} = (x, x, \dots, x)\in \mathcal{H}^r$, i.e., there exist $\kappa\geq 0$ and $\delta >0$ such that, for all $\mathbf{z}\in \mathbb{B}(\mathbf{x};\delta)$,
	\begin{equation*}
		d_{C\cap D}(\mathbf{z}) \leq \kappa\max\{d_C(\mathbf{z}), d_D(\mathbf{z})\}.
	\end{equation*}
	Let $z\in \mathbb{B}(x;\delta/\sqrt{r})$. Then $\mathbf{z} := (z, z, \dots, z)\in D\cap \mathbb{B}(\mathbf{x};\delta)$ and the above inequality implies $d_{C\cap D}(\mathbf{z}) \leq \kappa d_C(\mathbf{z})$. Thus, by using \eqref{e:dCD} and \eqref{e:dC}, we deduce linear regularity of $\{K_1, K_2, \dots, K_r\}$ around $x$.
	
	\ref{p:equi_strreg}: For all $\mathbf{x} = (x, x, \dots, x)\in D$, we have from \cite[Proposition~1.2]{mordukhovich} that 
	\begin{equation}\label{e:NC}
		N_C(\mathbf{x}) = N_{K_1}(x) \times N_{K_2}(x) \times \dots \times N_{K_r}(x)
	\end{equation}
	and from, e.g., \cite[Proposition~26.4(ii)]{bcombettes} that
	\begin{equation}\label{e:ND}
		N_D(\mathbf{x}) = \{(u_1, u_2, \dots, u_r)\in \mathcal{H}^r: u_1 + u_2 + \dots + u_r = 0\}.
	\end{equation}
	
	Assume that $\{K_1, K_2, \dots, K_r\}$ is strongly regular at $x\in \mathcal{H}$. Set $\mathbf{x} := (x, x, \dots, x)\in \mathcal{H}^r$ and let $\mathbf{u}\in N_C(\mathbf{x})\cap (-N_D(\mathbf{x}))$. In view of \eqref{e:NC}, we can write $\mathbf{u} = (u_1, u_2, \dots, u_r)$ with $u_j\in N_{K_j}(x)$. Since $\mathbf{u}\in -N_D(\mathbf{x})$, it follows from \eqref{e:ND} that $u_1 + u_2 + \dots + u_r = 0$. By the strong regularity of $\{K_1, K_2, \dots, K_r\}$, we have $u_1 = u_2 = \dots = u_r = 0$, and so $\mathbf{u} = 0$. Altogether, we have $N_C(\mathbf{x})\cap (-N_D(\mathbf{x})) = \{0\}$, and thus $\{C,D\}$ is strongly regular at $\mathbf{x}$.
	
	Conversely, assume that $\{C,D\}$ is strongly regular at $\mathbf{x} = (x, x, \dots, x)\in \mathcal{H}^r$ and assume that 
	\begin{equation*}
		u_1 + u_2 + \dots + u_r = 0 \text{~~with~~} u_j\in N_{K_j}(x). 
	\end{equation*}
	Then $\mathbf{u} := (u_1, u_2, \dots, u_r)\in N_C(\mathbf{x})$ due to \eqref{e:NC} and, in turn, \eqref{e:ND} implies that $\mathbf{u}\in -N_D(\mathbf{x})$, so $\mathbf{u}\in N_C(\mathbf{x})\cap (-N_D(\mathbf{x})) = \{0\}$. We therefore have that $u_1 = u_2 = \dots = u_r = 0$, which proves the strong regularity of $\{K_1, K_2, \dots, K_r\}$ at $x$.
\end{proof}

Recall that a sequence $(x_n)_{n\in \mathbb{N}}$ is said to \emph{converge $R$-linearly} to a point $x^\ast$ if there exist $\rho\in [0,1)$ and $\sigma >0$ such that, for all $n\in \mathbb{N}$,
\begin{equation*}
	\|x_n-x^\ast\| \leq \sigma\rho^n.
\end{equation*}

\begin{theorem}\label{t:noncvx}
	Let $K_1, K_2, \dots, K_{r-1}$ be closed subsets and $K_r$ be a closed affine subspace of $\mathcal{H}$ such that $\bigcap_{j=1}^r K_j \neq \varnothing$ and $P_{K_{r-1}}(K_r) \subseteq K_r$. Suppose that $K_1, \dots, K_{r-2}$, and $K_{r-1}\cap K_r$ are superregular at a point $\overline{x} \in \bigcap_{j=1}^r K_j$. Set $\overline{\mathbf{x}} := (\overline{x}, \overline{x}, \dots, \overline{x})\in \mathcal{H}^{r-1}$. Then the following statements hold.
	\begin{enumerate}[label =(\alph*)]
		\item\label{t:noncvx_MAP} 
		If $\{K_1, \dots, K_{r-2}, K_{r-1}\cap K_r\}$ is linearly regular around $\overline{x}$, then, whenever the starting point is sufficiently close to $\overline{\mathbf{x}}$, the sequence $(\mathbf{x}_n)_{n\in \mathbb{N}}$ generated by $\mathbf{x}_{n+1} \in S(\mathbf{x}_n)$ converges $R$-linearly to a point $\mathbf{x}^\ast = (x^\ast, x^\ast, \dots, x^\ast)\in V\cap W$ with $x^\ast\in \bigcap_{j=1}^r K_j$.
		\item\label{t:noncvx_DP} 
		If $\{K_1, \dots, K_{r-2}, K_{r-1}\cap K_r\}$ is strongly regular at $\overline{x}$, then, whenever the starting point is sufficiently close to $\overline{\mathbf{x}}$, the sequence $(\mathbf{x}_n)_{n\in \mathbb{N}}$ generated by $\mathbf{x}_{n+1} \in T(\mathbf{x}_n)$ converges $R$-linearly to a point $\mathbf{x}^\ast = (x^\ast, x^\ast, \dots, x^\ast)\in V\cap W$ with $x^\ast\in \bigcap_{j=1}^r K_j$.
	\end{enumerate}
\end{theorem}
\begin{proof}
	We first derive from Lemma~\ref{l:newproductspaceequivalence} that $\overline{\mathbf{x}} \in V\cap W$ and from Proposition~\ref{p:superregular}\ref{p:superregular_product} that $V$ is superregular at $\overline{\mathbf{x}}$. Since $P_{K_{r-1}}(K_r) \subseteq K_r$ and $K_r$ is a closed affine subspace, Theorem~\ref{t:main} implies that $P_{K_{r-1}}P_{K_r} = P_{K_{r-1}\cap K_r}$. In turn, $Q_V = P_V$.
	
	\ref{t:noncvx_MAP}: We have $S = P_WP_V$ and, by Proposition~\ref{p:equi}\ref{p:equi_linreg}, $\{V,W\}$ is linearly regular around $\overline{\mathbf{x}}$. Now, since $V$ is superregular at $\overline{\mathbf{x}}$ and $W$ is convex, applying \cite[Corollary~5.12(i)(b)]{dphan1} with $\lambda = \mu = \alpha = 1$, we get the conclusion.
	
	\ref{t:noncvx_DP}: According to Theorem~\ref{t:crDR}\ref{t:crDR_DR}, $T = (\I+R_VR_W)/2$. By Proposition~\ref{p:equi}\ref{p:equi_strreg}, $\{V,W\}$ is strongly regular at $\overline{\mathbf{x}}$. Noting that $V$ is superregular at $\overline{\mathbf{x}}$ and $W$ is convex, and using \cite[Corollary~5.12(i)(a)]{dphan1} with $\lambda = \mu = 2$ and $\alpha = 1/2$ (see also \cite[Theorem~4.3]{phan}), we complete the proof.
\end{proof}

\section{Application: Wavelet Construction}\label{sec:wavelets}

A \emph{wavelet} $\psi$ on the line is a function whose dyadic dilation and integer translations form an orthonormal basis for $L^2(\mathbb{R}, \mathbb{C})$. The utility of wavelets in analyzing and synthesizing signals relies on certain wavelet properties like\emph{ compact support} and \emph{regularity}. The earliest examples of compactly supported smooth wavelets with orthonormal shifts were first achieved by Daubechies \cite{daubechies} through the \emph{multiresolution analysis} (MRA) introduced by Mallat and Meyer \cite{mallat,meyer}. The methods employed by Daubechies are heavily reliant on complex analysis techniques that are not readily extendable to higher dimensions.

Recently, wavelet construction has been formulated as a feasibility problem \cite{franklin,fhtam,fhtamfull}. The product space DR and MAP, along with other projection algorithms, have been successfully employed to solve the wavelet feasibility problem. The product DR was observed to yield both already known and unseen examples of wavelets on the line consistently. This approach has also been extended to produce nonseparable wavelets on the plane which required a higher number of constraint sets. In certain applications in signal and image processing, the efficiency of these wavelets requires additional properties including \emph{real-valuedness}, \emph{symmetry}, and \emph{cardinality} \cite{dhlakey}.  Unfortunately, the inclusion of more constraints also requires additional product space dimensions. As the number of constraints gets large, the size of formulation becomes computationally intractable. It is on this ground that we want to evade an additional dimension by exploiting the property in \eqref{def:generalcondition} whenever it is viable.

We also remark that there are theoretical obstructions to obtain wavelets with the desired properties. Except for the case of Haar wavelet, there exists no symmetric, real-valued wavelets with orthonormal shifts, and compact support \cite{daubechies}. However, if we remove the real-valuedness condition, we may be able to obtain complex-valued scaling function and wavelet with perfect symmetry properties. Similarly, there exist no continuous, cardinal wavelets with compact support, and orthogonal shifts \cite{xia}. These theoretical obstructions may also be circumvented, without completely ruling out the desirable benefits of perfect symmetry or cardinality, by seeking for near-symmetry or near-cardinality \cite{dhlakey}.

In this section, we recall the wavelet feasibility problem and verify that a pair of its constraint sets satisfy \eqref{def:generalcondition}. For purposes of illustration, we set up feasibility problems for constructing real-valued smooth orthogonal wavelets, and for symmetric smooth orthogonal wavelets.  We use the constraint-reduced DR and MAP to solve the feasibility problems.

\subsection{The Wavelet Construction Problem}
Wavelet orthonormal bases are constructed by finding a scaling function--wavelet pair $(\phi, \psi)$, where $\phi$ comes from an MRA. This construction reduces to finding a matrix-valued function $U(\xi): \mathbb{R} \to \mathbb{C}^{2\times 2}$ of the form
\begin{equation*}
	U(\xi )=\left[\begin{matrix}m_0(\xi )&m_1(\xi )\\
		m_0(\xi +1/2)&m_1(\xi +1/2)\end{matrix}\right]\label{def:U}
\end{equation*}
where $m_0$ and $m_1$ are trigonometric series called \emph{filters} associated to the scaling function $\phi$ and wavelet $\psi$, respectively. Finding the coefficients of these filters is key to constructing a $(\phi, \psi)$ pair.

\subsubsection*{MRA Conditions and Design criteria}

A \emph{consistency condition} arises from the definition of $U(\xi)$, that is, $U(\xi +1/2)=\sigma U(\xi )$ where $\sigma$ is the ``row swap'' matrix. Additionally, a necessary condition for the orthonormality of the shifts and dilates of $\psi$ is that $m_0(0)=1$ and $U(\xi)$ is \emph{unitary almost everywhere}. For $\phi$ and $\psi$ to be compactly supported on $[0,M-1]$ for an even $M\geq 4$, we seek to impose that $m_0$ and $m_1$ be trigonometric polynomials of the form $m_0(\xi )=\sum_{k=0}^{M-1}h_ke^{2\pi ik\xi}$ and $m_1(\xi )=\sum_{k=0}^{M-1}g_ke^{2\pi ik\xi}$. Consequently, $U(\xi )=\sum_{k=0}^{M-1}A_ke^{2\pi ik\xi}$ with each $A_k\in{\mathbb C}^{2\times 2}$. The regularity criterion can be achieved by forcing $\frac{d^{\ell}}{d\xi^{\ell}}U(0)$ to be diagonal for all $0< \ell \leq D$, for some fixed $0<D\leq \frac{M-2}{2}$. Here, a higher value of $D$ would mean more regularity for the wavelet. To ensure that we obtain real-valued scaling and wavelet functions, we require $U(\xi)=\overline{U(-\xi)}$. Finally, if $U(\xi)^{\dagger}$ denotes a copy of $U(\xi)$ with negated off-diagonal entries and $\phi$ is symmetric about the center of support, then $U(\xi) = e^{2\pi i (M-1)\xi}U(\xi)^{\dagger}$.

\subsubsection*{Discretisation by Uniform Sampling}

The compact support condition allows us to write $U(\xi)$ as a matrix-valued trigonometric polynomial of degree $M-1$. And because a trigonometric polynomial of degree $M-1$ is determined by $M$ distinct points, we discretise $U(\xi)$ by a uniform sampling at $M$ points $\{\frac{j}{M}\}_{j=0}^{M-1} \subseteq [0,1)$. If $U_j = U(\frac{j}{M})$, then the sampling procedure produces an \emph{ensemble} $\mathcal{U}=(U_0,U_1,\dots, U_{M-1}) \in (\mathbb{C}^{2\times 2})^M$  of matrices. Moreover, the coefficient matrices $A_k$ may be obtained from the ensembles by an $M$-point discrete Fourier transform, that is, 
\begin{equation}
	A_k=({\mathcal F}_M{\mathcal U})_k=\dfrac{1}{M}\sum_{j=0}^{M-1}U_je^{-2\pi ijk/M},\label{A_k}
\end{equation}
which is also invertible to recover back 
$U_j=({\mathcal F}_M^{-1}{\mathcal A})_j$.
This establishes a connection between the uniform samples and the coefficient matrices $A_k$ of $U(\xi)$.

\subsubsection*{Wavelet Properties Encoded on the Ensembles}

The consistency condition is imposed on the ensemble of samples to satisfy $U_{j +\frac{M}{2}} = \sigma U_j$ for all $j \in \{0,1,\dots, M-1\}$. On the other hand, unitarity of each sample $U_j=U(\frac{j}{M})$ for $j \in \{0,1,\dots,M-1\}$ is insufficient to ensure the unitarity of $U(\xi )$ almost everywhere. However, it transpires that forcing $U(\xi)$ to be unitary at $2M$ samples, uniformly chosen to be $U(\frac{j}{2M})$ and $U(\frac{2j+1}{2M})$, for $j \in \{0,1,\dots,M-1\}$, is sufficient for $U(\xi)$ to be unitary almost everywhere. Incidentally, given $\mathcal{U} = (U(\frac{j}{M}))_{j=0}^{M-1}$, the other $M$ samples written to form an ensemble $\tilde{\mathcal{U}}$ may be obtained from $\mathcal{U}$ using $\tilde{\mathcal U}={\mathcal F}_M^{-1}\chi_M{\mathcal F}_M({\mathcal U})$, where $(\chi_M)_j=e^{\pi ij/M}$ for $j=\{0,1,\dots, M-1\}$. In terms of the sample matrices $U_j$, the regularity condition is imposed by forcing $\sum_{j=0}^{M-1}j^\ell A_j$ to be diagonal, where 
\begin{equation*}
	\sum_{j=0}^{M-1}j^\ell A_j=\frac{1}{M}\sum_{k =0}^{M-1}\alpha_{\ell k}U_k 
	\text{~~and~~}
	\alpha_{\ell k}=\dfrac{1}{M}\sum_{j=0}^{M-1}j^\ell e^{-2\pi ik j/M}.
\end{equation*}
For real-valuedness, the ensembles must satisfy $U_j=\overline{U_{M-j}}$ for $j \in \{1,2,\dots, \frac{M}{2}\}$. Lastly, we require $U_j = e^{2\pi i (M-1)j/M}U^{\dagger}_{M-j}$ for all $j \in \{1,2,\dots, \frac{M}{2}\}$ to meet the symmetry condition.

\subsubsection*{Wavelet Construction as a Feasibility Problem}

Let $(\mathbb{C}^{2\times 2})_{\sigma}^M$ denote the collection of ensembles in $(\mathbb{C}^{2\times 2})^M$ that satisfy the consistency condition. Further, let $\mathbb{U}(2)$ denote the collection of all $2$-by-$2$ unitary matrices. For an even $M\geq 4$ and a fixed $0<D\leq \frac{M-2}{2}$, we define $C_1, C_2,C_3, C_4^{(R)},C_4^{(S)} \subseteq ({\mathbb C}^{2\times 2})_\sigma^M$ as follows.	
\begin{subequations}\label{eq:waveletconstraints}
	\begin{align}
		C_1 &:= \left\{\mathcal{U}:\, U_0=\begin{bmatrix}1&0\\0&z\\ \end{bmatrix},\,|z|=1,\, U_j\in {\mathbb U}(2),\,j\in \{0,1,\dots, M/2\}\right\},  \\
		C_2 &:= \left\{\mathcal{U}:\, ({\mathcal F}_M\chi_M({\mathcal F}_M)^{-1}(\mathcal{U}))_j\in {\mathbb U}(2),\, j\in \{0,1,\dots, M/2\}\right\},  \\
		C_3 &:= \left\{\mathcal{U}:\, \sum_{k=0}^{M-1}\alpha _{\ell k}U_k\in\text{diag\,}({\mathbb C}^{2\times 2}),\ 1\leq\ell\leq D\right\},\\
		C_4^{(R)} &:= \left\{\mathcal{U}: U_j = \overline{U_{M-j}}, \, j \in \{1,2,\dots, M/2\} \right\},\\
		C_4^{(S)}&:= \left\{\mathcal{U}: U_j = e^{2\pi i (M-1)j/M}U^{\dagger}_{M-j}, \, j \in \{1,2,\dots, M/2\} \right\}.
	\end{align}
\end{subequations}

\begin{problem}[Symmetric wavelets]\label{prob:waveletproblemsymm}
	The problem to construct symmetric smooth orthogonal wavelet is to find an ensemble $\ \mathcal{U}=(U_0,\dots,U_{M-1})\in \bigcap_{k=1}^3 C_k \cap C_4^{(S)} \subseteq ({\mathbb C}^{2\times 2})_\sigma^M$.
\end{problem}

\begin{problem}[Real-valued wavelets]\label{prob:waveletproblemreal}
	The problem to construct real-valued smooth orthogonal wavelet is to find an ensemble $\ \mathcal{U}=(U_0,\dots,U_{M-1})\in \bigcap_{k=1}^3 C_k \cap C_4^{(R)} \subseteq ({\mathbb C}^{2\times 2})_\sigma^M$.
\end{problem}

Note that before the constraint sets are defined, the parameters $M$ and $D$ must be chosen first. A particular combination of values of $M$ and $D$ corresponds to a specific case of Problem~\ref{prob:waveletproblemsymm} or Problem~\ref{prob:waveletproblemreal}. We also remark that $C_1$ and $C_2$ are nonconvex subsets of $({\mathbb C}^{2\times 2})_\sigma^M$, and every ensemble in both $C_1$ and $C_2$ will satisfy the unitarity condition. The subspaces $C_3$, $C_4^{(R)}$, and $C_4^{(S)}$ are constraint sets for regularity, real-valuedness, and symmetry, respectively. The projectors onto $C_1$, $C_2$, and $C_3$ are computed in \cite[Section~6.3]{franklin} and those onto $C_4^{(R)}$ and $C_4^{(S)}$ are referred to \cite[Section~3]{dhlakey}. We will show that $C_4^{(R)}$ and $C_4^{(S)}$ are both invariant under the projector onto $C_1$ which we recall in the next proposition. 

\begin{proposition}\label{proj:C1}
	Let $\mathcal{U} = (U_0,U_1,\dots, U_{M-1}) \in (\mathbb{C}^{2\times 2})_{\sigma}^M$ and $\tilde{\mathcal{U}}:=\{\tilde{U}_0,\tilde{U_1},\dots, \tilde{U}_{M-1}\} \in P_{C_1}(\mathcal{U})$. Suppose further that $z$ is the $(2,2)$-entry of $U_0$ and that $U_j = X_j \Sigma_j Y_j^{\ast}$ is a singular value decomposition for $U_j$ where $j\in \left\{1,2,\dots,M-1\right\}\backslash\left\{\frac{M}{2}\right\}$. Then 
	\begin{align*}
		\tilde{U}_0 &= \begin{bmatrix}
			1 & 0 \\ 
			0 & \frac{z}{|z|}  \\
		\end{bmatrix},\quad 
		\tilde{U}_{\frac{M}{2}} = \sigma \tilde{U}_0, 
		\text{~~and~~}\tilde{U}_j =X_jY_j^{\ast}
	\end{align*}
	for $j \in \left\{1,2,\dots,M-1\right\}\backslash \left\{\frac{M}{2}\right\}$.
\end{proposition}
\begin{proof}
	See \cite[Lemma~6.3.4]{franklin}.
\end{proof}

We emphasize that the ensembles in $P_{C_1}(\mathcal{U})$ do satisfy the consistency condition \cite[Lemma~6.3.6]{franklin}. We now verify two important relations among the constraint sets. These relations give us appropriate pairs of constraint sets for applying the constraint reduction reformulation to the wavelet feasibility problem. 

\begin{theorem}\label{t:PC1}
	Let $C_1$, $C_4^{(R)}$, and $C_4^{(S)}$ be as defined in \eqref{eq:waveletconstraints}. Then the following statements hold.
	\begin{enumerate}[label =(\alph*)]
		\item\label{t:PC1_C4R} 
		$P_{C_1}\left(C_4^{(R)}\right) \subseteq C_4^{(R)}$.
		\item\label{t:PC1_C4S} 
		$P_{C_1}\left(C_4^{(S)}\right) \subseteq C_4^{(S)}$.
	\end{enumerate}
\end{theorem}
\begin{proof}
	\ref{t:PC1_C4R}: Let $\mathcal{U}\in C_4^{(R)}$ and $\tilde{\mathcal{U}} \in P_{C_1}(\mathcal{U})$. Then $U_j = \overline{U_{M-j}}$ and $U_j = \sigma U_{j+\frac{M}{2}}$ for $j \in \{0,1,\dots, \frac{M}{2}\}$. Consequently, $U_0$ and $U_{\frac{M}{2}}$ have real entries. We deduce from Proposition~\ref{proj:C1} that $\tilde{U}_0$ and $\tilde{U}_{\frac{M}{2}}$ will also have real entries. For $j \in \left\{1,2,\dots,\frac{M}{2}-1\right\}$, again by Proposition~\ref{proj:C1}, $\tilde{U}_j = X_jY_j^{\ast}$, where $U_j = X_j\Sigma_jY_j^{\ast}$ is a singular value decomposition for $U_j$. Moreover,
	\begin{equation*}
		U_{M-j} = \overline{U_j} =\overline{X_j\Sigma_jY_j^{\ast}} = \overline{X_j}\Sigma_j\overline{Y_j^{\ast}}
	\end{equation*}
	is a singular value decomposition for $U_{M-j}$, and so $\tilde{U}_{M-j} = \overline{X_jY_j^{\ast}} = \overline{\tilde{U}_j}$. Therefore, $\tilde{\mathcal{U}} \in C_4^{(R)}$.
	
	\ref{t:PC1_C4S}: Let $\mathcal{U}\in C_4^{(S)}$ and $\tilde{\mathcal{U}} \in P_{C_1}(\mathcal{U})$. Then $U_j = e^{2\pi i (M-1)j/M}U_{M-j}^{\dagger}$ and $U_j = \sigma U_{j+\frac{M}{2}}$ for $j \in \{0,1,\dots, \frac{M}{2}\}$. In particular, $U_0 = U_0^{\dagger}$ and $U_{\frac{M}{2}} = -U_{\frac{M}{2}}^{\dagger}$. We know from Proposition~\ref{proj:C1} that $\tilde{U}_0$ is diagonal and deduce that $\tilde{U}_0 = \tilde{U}_0^{\dagger}$ and $\tilde{U}_{\frac{M}{2}} = \sigma\tilde{U}_0 = -(\sigma \tilde{U}_0)^{\dagger} = -\tilde{U}_{\frac{M}{2}}^{\dagger}$. For $j \in \left\{1,2,\dots,\frac{M}{2}-1\right\}$, we also learn from Proposition~\ref{proj:C1} that $\tilde{U}_j = X_jY_j^{\ast}$, where $U_j = X_j\Sigma_jY_j^{\ast}$ is a singular value decomposition for $U_j$. Denote $s = e^{-2\pi i (M-1)j/M}$ and $\tau = \text{diag}(-1,1) \in \mathbb{C}^{2\times 2}$. Then
	\begin{align*}
		U_{M-j} = e^{-2\pi i (M-1)j/M} U_j^{\dagger}
		= s\tau U_j \tau^{\ast}
		= (s\tau  X_j)\Sigma_j(Y_j^{\ast} \tau^{\ast})
	\end{align*}
	is a singular value decomposition for $U_{M-j}$ since $|s|=1$ and $\tau$ is unitary. Hence, 
	\begin{align*}
		\tilde{U}_{M-j} = (s\tau X_j) (Y_j^{\ast}\tau^{\ast})
		= s(\tau X_jY_j^{\ast}\tau^{\ast})
		= s( X_jY_j^\ast)^{\dagger}
		= e^{-2\pi i (M-1)j/M} \tilde{U}_j^{\dagger},
	\end{align*}
	and we deduce that $\tilde{\mathcal{U}} \in C_4^{(S)}$.
\end{proof}

The results in Theorem~\ref{t:PC1} further justify our choice of $C_1$ and $C_4^{(S)}$ as the pair of constraints to replace with their intersection for constraint reduction reformulation of Problem~\ref{prob:waveletproblemsymm}. Similarly, the pair of $C_1$ and $C_4^{(R)}$ is the natural choice for Problem~\ref{prob:waveletproblemreal}. We will solve these problems in the next subsection. 

We note that a solution $\mathcal{U}=(U_0,\dots,U_{M-1})$ of Problem \ref{prob:waveletproblemsymm} or \ref{prob:waveletproblemreal} contains the $M$ samples of $U(\xi)$ from which we recover the coefficients $A_k$ using \eqref{A_k}. Consequently, the coefficients of the scaling filter $m_0(\xi)$ and wavelet filter $m_1(\xi)$ may be easily pulled out from the $A_k$'s. Through the \emph{cascade algorithm} applied to the coefficients of $m_0$ and $m_1$, we may be able to plot the scaling function $\phi$ and wavelet $\psi$, respectively.

\subsection{Numerical Experiments}

The wavelet feasibility problems defined in Problems \ref{prob:waveletproblemsymm}--\ref{prob:waveletproblemreal} can be straightforwardly reformulated to a two-set feasibility problem using the product space reformulation defined in Section \ref{sec:productspacetrick}. The product DR and MAP are then employable to solve the two-set problem. Alternatively, we may apply the constraint-reduction reformulation to the problems at hand. We abuse notation by consistently denoting the reduced product space constraints as in Definition \ref{constraintreduction} for both problems.\\

\noindent \textit{Constraint-reduction reformulation for Problem~\ref{prob:waveletproblemsymm}}: The product space constraints for obtaining symmetric wavelets are defined by
\begin{align*}
	V &:=\left(C_1\cap C_4^{(S)}\right)\times C_2 \times C_3 \subseteq \left((\mathbb{C}^{2\times 2})_{\sigma}^M\right)^3,\\
	W &:= \left\{(\mathcal{U}_j)_{j=1}^{3} \in \left((\mathbb{C}^{2\times 2})_{\sigma}^M\right)^3: \mathcal{U}_1 = \mathcal{U}_2 = \mathcal{U}_3\right\}.
\end{align*}

\noindent \textit{Constraint-reduction reformulation for Problem~\ref{prob:waveletproblemreal}}: The product space constraints for obtaining real-valued wavelets are defined by
\begin{align*}
	V &:=\left(C_1\cap C_4^{(R)}\right)\times C_2 \times C_3 \subseteq \left((\mathbb{C}^{2\times 2})_{\sigma}^M\right)^3,\\
	W &:= \left\{(\mathcal{U}_j)_{j=1}^{3} \in \left((\mathbb{C}^{2\times 2})_{\sigma}^M\right)^3: \mathcal{U}_1 = \mathcal{U}_2 = \mathcal{U}_3\right\}.
\end{align*}

The associated operators $Q_V$ and $P_W$ for both the two new problems are defined similar to what appeared in Definition \ref{constraintreduction}.

For constructing symmetric wavelets, we will solve two cases of Problem~\ref{prob:waveletproblemsymm} where $(M,D)=(6,2)$ and $(M,D)=(6,1)$. Similarly, for real-valued wavelets, we work out two cases of Problem~\ref{prob:waveletproblemreal} corresponding to the parameters $(M,D)=(6,2)$ and $(M,D)=(6,1)$. For each particular problem, we employ product DR, constraint-reduced DR, product MAP, and constraint-reduced MAP. However, we only compare the performance of product DR against constraint-reduced DR, and the performance of product MAP against constraint-reduced MAP. Henceforth, we let $(x_n)_{n\in \mathbb{N}}$ be the sequence of iterates generated by a projection algorithm. We will employ a particular algorithm to a problem twice using two different tolerance values, namely, $\varepsilon =10^{-6}$ and $\varepsilon =10^{-9}$. For the DR variant, we use the stopping criterion given by $\|Q_VP_W(x_n) - P_W(x_n)\| < \varepsilon$ which when satisfied indicates that $P_W(x_n)$ can be declared as a feasible point. Similarly for constraint-reduced MAP, we set a stopping criterion $\|Q_V(x_n) - x_n\| <\varepsilon$ to decide that the iterate $x_n$ lies on the intersection of $V$ and $W$. We consider a projection algorithm to have \emph{solved} our feasibility problem if and when it attains a point that satisfies the stopping criterion within the cutoff of $50,000$ iterates. For our numerical results, we provide statistics on the \emph{number of iterations} which we mainly consider as performance measure. We also look at the average running time of an algorithm in solving a particular problem. Additionally, we comment on the versatility of an algorithm in tackling the nonconvex wavelet feasibility problem by counting the number of times it solves a particular problem, initialized at $1,000$ ensembles that satisfy the consistency condition and with complex entries having real and imaginary parts chosen from uniformly distributed random number in the interval $(0,1)$. All datasets generated and analysed in this study are available from the corresponding author on request.

\subsubsection*{Symmetric Wavelets}

In constructing symmetric wavelets, we solve Problem~\ref{prob:waveletproblemsymm}. We employ the product DR and MAP, and their constraint-reduced variants to solve the product space and constraint-reduced versions of Problem \ref{prob:waveletproblemsymm} with $(M,D)=(6,2)$ and $(M,D)=(6,1)$. We only compare the performance of product DR with that of constraint-reduced DR, doing the same for MAP. In this way, we are essentially comparing the robustness of the product space  and constraint reduction reformulations.

Table~\ref{tab:symmDR} summarizes the performance of DR in solving Problem~\ref{prob:waveletproblemsymm} using product space and constraint reduction reformulations. In all versions of the problem considered, the constraint-reduced DR solved every test case while product DR failed in a number of cases. In instances where both algorithms converged, the constraint-reduced DR used up lesser number of iterations in at least 78\% of the time. This suggests that the constraint-reduced DR outperforms product DR as also reflected in the computed mean and median number of iterations. The average running time (in seconds) of the constraint-reduced DR is also better than that of the product DR.

\begin{table}[h!]
	\centering
	\resizebox{\textwidth}{!}{%
	\begin{tabular}{@{}cccccccccc@{}}
		\toprule
		\multirow{2}{*}{\textbf{\begin{tabular}[c]{@{}c@{}}Problem\\ parameters\end{tabular}}} & \multirow{2}{*}{\textbf{$\varepsilon$}} & \multirow{2}{*}{\textbf{algorithm}} & \multirow{2}{*}{\textbf{\begin{tabular}[c]{@{}c@{}}cases\\ solved\end{tabular}}} & \multirow{2}{*}{\textbf{\begin{tabular}[c]{@{}c@{}}solved\\ alone\end{tabular}}} & \multirow{2}{*}{\textbf{\begin{tabular}[c]{@{}c@{}}solved \\ by both\end{tabular}}} & \multicolumn{4}{c}{\textbf{when solved by both}} \\ \cmidrule(l){7-10} 
		&  &  &  &  &  & \textbf{wins} & \textbf{mean} & \textbf{median} & \textbf{running time} \\ \midrule
		\multirow{4}{*}{M=6, D=2} & \multirow{2}{*}{$10^{-6}$} & P--DR & 940 & 0 & 940 & 187 & 3862 & 3393 & 8.6575 \\
		&  & CR--DR & 1000 & 60 & 940 & 753 & 3223 & 2735 & 7.1764 \\ \cmidrule(l){2-10} 
		& \multirow{2}{*}{$10^{-9}$} & P--DR & 956 & 0 & 956 & 206 & 6969 & 5941 & 15.0907 \\
		&  & CR--DR & 1000 & 44 & 956 & 750 & 5439 & 4754 & 11.6524 \\ \midrule
		\multirow{4}{*}{M=6, D=1} & \multirow{2}{*}{$10^{-6}$} & P--DR & 934 & 0 & 934 & 148 & 4076 & 3456 & 8.8109 \\
		&  & CR--DR & 1000 & 66 & 934 & 786 & 3073 & 2724 & 6.5987 \\ \cmidrule(l){2-10} 
		& \multirow{2}{*}{$10^{-9}$} & P--DR & 934 & 0 & 934 & 154 & 7262 & 5988 & 16.1757 \\
		&  & CR--DR & 1000 & 66 & 934 & 780 & 5378 & 4721 & 11.8203 \\ \bottomrule
	\end{tabular}}
	\caption{Statistics on the performance of product DR (P--DR) and constraint-reduced DR (CR--DR) for wavelet feasibility problems with symmetry constraint.}
	\label{tab:symmDR}
\end{table}

Similarly, Table~\ref{tab:symmMAPs} highlights our results for cases where MAP is used to solve the feasibility problem through product space and constraint reduction reformulations. In our statistics, the two algorithms solved all test cases with constraint-reduced MAP incurring lesser number of iterations in at least 97\% of the time. This suggests that constraint-reduced MAP outperforms product MAP in this sense as can also be seen in the computed mean and median number of iterations for both algorithms. Moreover, constraint-reduced MAP has exhibited a consistently favorable running time as compared to product MAP.

\begin{table}[h!]
	\centering
	\resizebox{\textwidth}{!}{%
	\begin{tabular}{@{}cccccccccc@{}}
		\toprule
		\multirow{2}{*}{\textbf{\begin{tabular}[c]{@{}c@{}}Problem\\ parameters\end{tabular}}} & \multirow{2}{*}{\textbf{$\varepsilon$}} & \multirow{2}{*}{\textbf{algorithm}} & \multirow{2}{*}{\textbf{\begin{tabular}[c]{@{}c@{}}cases\\ solved\end{tabular}}} & \multirow{2}{*}{\textbf{\begin{tabular}[c]{@{}c@{}}solved\\ alone\end{tabular}}} & \multirow{2}{*}{\textbf{\begin{tabular}[c]{@{}c@{}}solved \\ by both\end{tabular}}} & \multicolumn{4}{c}{\textbf{when solved by both}} \\ \cmidrule(l){7-10} 
		&  &  &  &  &  & \textbf{wins} & \textbf{mean} & \textbf{median} & \textbf{running time} \\ \midrule
		\multirow{4}{*}{M=6, D=2} & \multirow{2}{*}{$10^{-6}$} & P--MAP & 1000 & 0 & 1000 & 26 & 3337 & 3474 & 3.0885 \\
		&  & CR--MAP & 1000 & 0 & 1000 & 974 & 2521 & 2599 & 2.2512 \\ \cmidrule(l){2-10} 
		& \multirow{2}{*}{$10^{-9}$} & P--MAP & 1000 & 0 & 1000 & 2 & 5528 & 5648 & 4.5243 \\
		&  & CR--MAP & 1000 & 0 & 1000 & 998 & 4157 & 4232 & 3.3080 \\ \midrule
		\multirow{4}{*}{M=6, D=1} & \multirow{2}{*}{$10^{-6}$} & P--MAP & 1000 & 0 & 1000 & 23 & 3389 & 3477 & 2.4394 \\
		&  & CR--MAP & 1000 & 0 & 1000 & 977 & 2516 & 2600 & 1.7688 \\ \cmidrule(l){2-10} 
		& \multirow{2}{*}{$10^{-9}$} & P--MAP & 1000 & 0 & 1000 & 4 & 5569 & 5655 & 4.1696 \\
		&  & CR--MAP & 1000 & 0 & 1000 & 996 & 4149 & 4232 & 3.0308 \\ \bottomrule
	\end{tabular}}
	\caption{Statistics on the performance of product MAP (P--MAP) and constraint-reduced MAP (CR--MAP) for wavelet feasibility problems with symmetry constraint.}
	\label{tab:symmMAPs}
\end{table}

It is also noteworthy that based on our statistics, MAP's variants are more effective than DR's in finding symmetric wavelets. Figure~\ref{fig:symmetricexample} shows an example of a symmetric scaling function and an anti-symmetric wavelet generated by solving Problem \ref{prob:waveletproblemsymm} with $(M,D)=(6,2)$. 

\begin{figure}[h!]
	\centering
	\subfloat[A symmetric scaling function.\label{fig:symmetricscaling}]{\includegraphics[width = 0.4\linewidth]{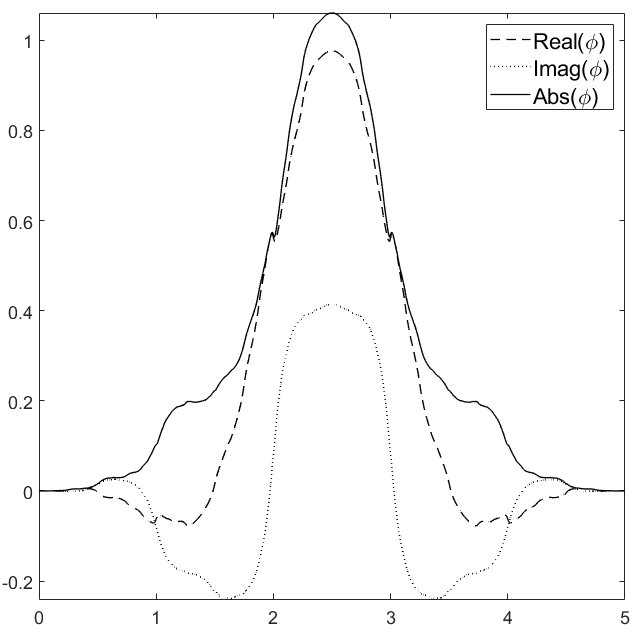}}\hspace{0.5cm}
	\subfloat[An anti-symmetric wavelet.\label{fig:symmetricwavelet}]{\includegraphics[width = 0.4\linewidth]{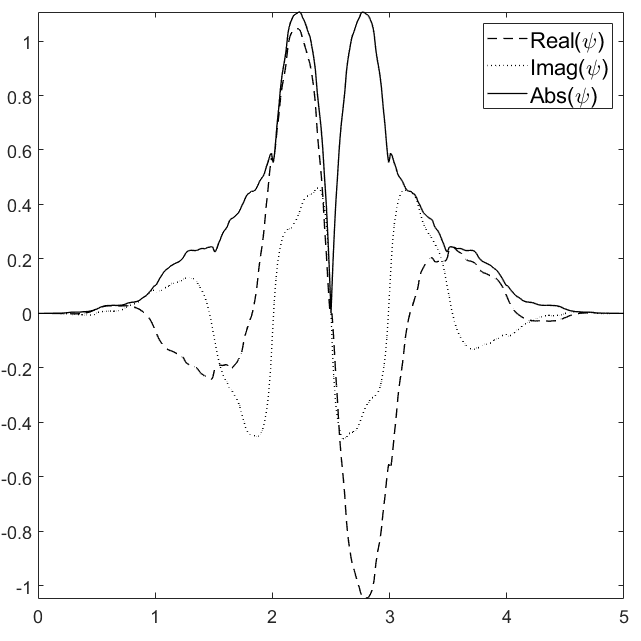}}\\
	\caption{Complex-valued compactly supported smooth scaling function and wavelet with symmetry properties obtained by solving Problem~\ref{prob:waveletproblemsymm} with $(M,D)=(6,2)$.}  \label{fig:symmetricexample}
\end{figure}

\subsubsection*{Real-valued Wavelets}

To construct real-valued wavelets, we need to deal with Problem~\ref{prob:waveletproblemreal}. We employ the product DR and MAP, and their constraint-reduced variants in the two problems where $(M,D)=(6,2)$ and $(M,D)=(6,1)$.

Table~\ref{tab:realDR} shows the performance of DR in solving Problem~\ref{prob:waveletproblemreal} using product space and constraint reduction reformulations. For the particular problem where $(M,D)=(6,2)$ and $\varepsilon = 10^{-9}$, the product DR solved more cases than constraint-reduced DR. Nevertheless, in cases where both algorithms solved the feasibility problem, the constraint-reduced DR used up lesser number of iterations 97\% of the time. This suggests that the constraint-reduced DR outperforms product DR in terms of the number of iterations. This claim is supported by the computed mean and median number of iterations for the contraint-reduced DR that are less than that of the product DR. Moreover, the average running time (in seconds) of constraint-reduced DR is better than product DR. Similar results are observed for the problem where $(M,D)=(6,2)$ and $\varepsilon = 10^{-6}$. For the problems where $(M,D)=(6,1)$ with the tolerance values $10^{-6}$ and $10^{-9}$, the constraint-reduced DR outperforms the product version in terms of number of iterations and running time.

\begin{table}[h!]
	\centering
	\resizebox{\textwidth}{!}{%
	\begin{tabular}{@{}cccccccccc@{}}
		\toprule
		\multirow{2}{*}{\textbf{\begin{tabular}[c]{@{}c@{}}Problem\\ parameters\end{tabular}}} & \multirow{2}{*}{\textbf{$\varepsilon$}} & \multirow{2}{*}{\textbf{algorithm}} & \multirow{2}{*}{\textbf{\begin{tabular}[c]{@{}c@{}}cases\\ solved\end{tabular}}} & \multirow{2}{*}{\textbf{\begin{tabular}[c]{@{}c@{}}solved\\ alone\end{tabular}}} & \multirow{2}{*}{\textbf{\begin{tabular}[c]{@{}c@{}}solved \\ by both\end{tabular}}} & \multicolumn{4}{c}{\textbf{when solved by both}} \\ \cmidrule(l){7-10} 
		&  &  &  &  &  & \textbf{wins} & \textbf{mean} & \textbf{median} & \textbf{running time} \\ \midrule
		\multirow{4}{*}{M=6, D=2} & \multirow{2}{*}{$10^{-6}$} & P--DR & 619 & 262 & 357 & 21 & 771 & 584 & 1.5704 \\
		&  & CR--DR & 497 & 140 & 357 & 336 & 618 & 457 & 1.2569 \\ \cmidrule(l){2-10} 
		& \multirow{2}{*}{$10^{-9}$} & P--DR & 619 & 262 & 357 & 11 & 1147 & 959 & 1.8231 \\
		&  & CR--DR & 497 & 140 & 357 & 346 & 1053 & 740 & 1.6551 \\ \midrule
		\multirow{4}{*}{M=6, D=1} & \multirow{2}{*}{$10^{-6}$} & P--DR & 827 & 334 & 493 & 64 & 560 & 301 & 0.8113 \\
		&  & CR--DR & 581 & 88 & 493 & 429 & 252 & 183 & 0.3615 \\ \cmidrule(l){2-10} 
		& \multirow{2}{*}{$10^{-9}$} & P--DR & 827 & 334 & 493 & 72 & 693 & 467 & 1.1542 \\
		&  & CR--DR & 581 & 88 & 493 & 421 & 358 & 283 & 0.5914 \\ \bottomrule
	\end{tabular}}
	\caption{Statistics on the performance of product DR (P--DR) and constraint-reduced DR (CR--DR) for wavelet feasibility problems with real-valuedness constraint.}
	\label{tab:realDR}
\end{table}

Similarly, Table~\ref{tab:realMAPs} summarizes our results when MAP is used to solve the feasibility problem through product space  and constraint reduction reformulations. In our statistics for both problems corresponding to $(M,D)=(6,2)$ and $(M,D)=(6,1)$ under two different values for $\varepsilon$, the two algorithms performed closely in terms of their efficacy to solve the feasibility problem. For the problem with $(M,D)=(6,2)$, product MAP solved a few more problems than the constraint-reduced version. However, when $(M,D)=(6,1)$, constraint-reduced MAP solved more cases than product DR. In cases where both algorithms solved the feasibility problem, the constraint-reduced MAP consistently rendered lesser number of iterations, outperforming the product MAP. These are reflected in the computed mean and median number of iterations for both algorithms. Constraint-reduced MAP also exhibited a consistently favorable running time.

\begin{table}[h!]
	\centering
	\resizebox{\textwidth}{!}{%
	\begin{tabular}{@{}cccccccccc@{}}
		\toprule
		\multirow{2}{*}{\textbf{\begin{tabular}[c]{@{}c@{}}Problem\\ parameters\end{tabular}}} & \multirow{2}{*}{\textbf{$\varepsilon$}} & \multirow{2}{*}{\textbf{algorithm}} & \multirow{2}{*}{\textbf{\begin{tabular}[c]{@{}c@{}}cases\\ solved\end{tabular}}} & \multirow{2}{*}{\textbf{\begin{tabular}[c]{@{}c@{}}solved\\ alone\end{tabular}}} & \multirow{2}{*}{\textbf{\begin{tabular}[c]{@{}c@{}}solved \\ by both\end{tabular}}} & \multicolumn{4}{c}{\textbf{when solved by both}} \\ \cmidrule(l){7-10} 
		&  &  &  &  &  & \textbf{wins} & \textbf{mean} & \textbf{median} & \textbf{running time} \\ \midrule
		\multirow{4}{*}{M=6, D=2} & \multirow{2}{*}{$10^{-6}$} & P--MAP & 264 & 95 & 169 & 0 & 355 & 354 & 0.2621 \\
		&  & CR--MAP & 235 & 66 & 169 & 169 & 264 & 263 & 0.1890 \\ \cmidrule(l){2-10} 
		& \multirow{2}{*}{$10^{-9}$} & P--MAP & 264 & 95 & 169 & 0 & 543 & 542 & 0.4549 \\
		&  & CR--MAP & 235 & 66 & 169 & 169 & 404 & 403 & 0.3294 \\ \midrule
		\multirow{4}{*}{M=6, D=1} & \multirow{2}{*}{$10^{-6}$} & P--MAP & 112 & 32 & 80 & 1 & 90 & 88 & 0.0731 \\
		&  & CR--MAP & 158 & 78 & 80 & 79 & 66 & 63 & 0.0525 \\ \cmidrule(l){2-10} 
		& \multirow{2}{*}{$10^{-9}$} & P--MAP & 112 & 32 & 80 & 1 & 133 & 129 & 0.1059 \\
		&  & CR--MAP & 158 & 78 & 80 & 79 & 97 & 93 & 0.0757 \\ \bottomrule
	\end{tabular}}
	\caption{Statistics on the performance of product MAP (P--MAP) and constraint-reduced MAP (CR--MAP) for wavelet feasibility problems with real-valuedness constraint.}
	\label{tab:realMAPs}
\end{table}

In contradistinction, MAP is not as robust as DR in solving the two cases of Problem~\ref{prob:waveletproblemreal} that we have considered, as suggested by the total number of test runs that MAP and DR solved.  Figure~\ref{fig:realexample} shows an example of real-valued scaling function-wavelet pair generated by solving Problem \ref{prob:waveletproblemreal} with $(M,D)=(6,2)$. This wavelet is exactly Daubechies' $_3\psi$ wavelet which is known to have the maximal number of vanishing moments for its length of support \cite[Chapter~5]{daubechies}. Other solutions may be obtained by lowering the requirement on regularity as in the case where $(M,D)=(6,1)$.

\begin{figure}[h!]
	\centering
	\subfloat[A real-valued scaling function.\label{fig:realscaling}]{\includegraphics[width = 0.4\linewidth]{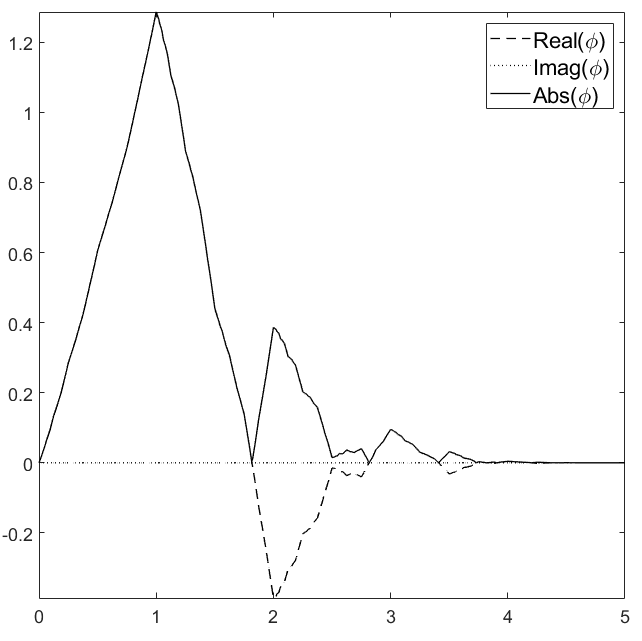}}\hspace{0.5cm}
	\subfloat[A real-valued wavelet.\label{fig:realwavelet}]{\includegraphics[width = 0.4\linewidth]{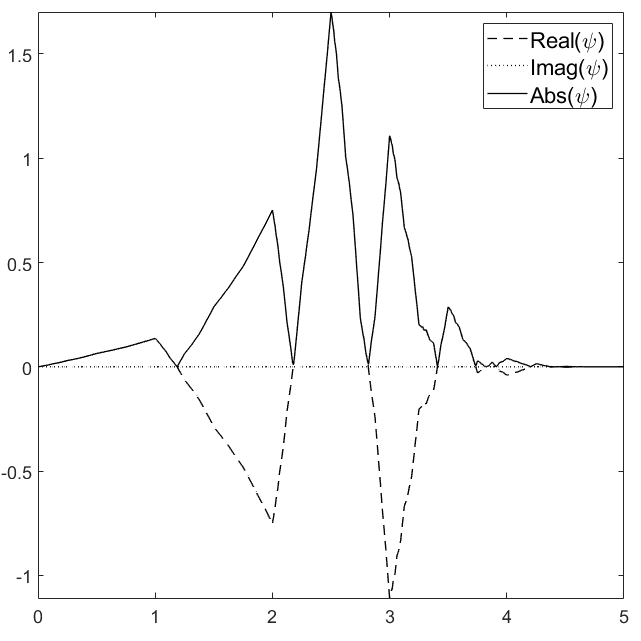}}\\
	\caption{Real-valued compactly supported smooth scaling function and wavelet obtained by solving Problem~\ref{prob:waveletproblemreal} with $M=6$ and $D=2$. These coincide with Daubechies' $_3\phi$ scaling function and $_3\psi$ wavelet.}  \label{fig:realexample}
\end{figure}

\section{Conclusions}

We have introduced a constraint reduction reformulation for converting many-set feasibility problems into two-set problems. It provides an equivalent formulation of many-set feasibility problems by replacing a pair of its constraint sets with their intersection, before applying Pierra's classical product space reformulation. Our new reformulation gives rise to constraint-reduced variants of any projection algorithm that can be used to solve two-set feasibility problems. We have presented a global convergence analysis for the constraint-reduced variants of DR and MAP in the convex setting, and a local convergence analysis in a nonconvex setting. In carrying out the analysis for the constraint-reduced DR, we have generalized a well-known result which guarantees that the composition of two projectors onto subspaces is again a projector onto the intersection. Even when the constraint sets do not possess the additional structure required, the constraint-reduced variants of projection algorithms still serve as useful heuristics for solving nonconvex feasibility problems.

The required property among the constraint sets for the convergence of constraint-reduced DR appear exactly in the wavelet feasibility problems so it provided us a suitable venue for numerical implementations of the new reformulation technique. In certain cases, the performance of constraint-reduced DR and MAP has been seen as improvement over their usual product variants.

\subsection*{Acknowledgments}
The authors would like to thank Scott Lindstrom for his helpful insights on Theorem~\ref{t:main}, and Hui Ouyang for her constructive inputs. The authors are also grateful to the reviewers for their valuable feedback and insightful comments. The authors were partially supported by the Australian Research Council through grants DP160101537 (MND, NDD, and JAH), DP190100555 (MND), and DE200100063 (MKT).

\bibliographystyle{siam}
\bibliography{refs}

\begin{thebibliography}{10}

\bibitem{abtam2}
{\sc F.~J. Arag\'on~Artacho, J.~M. Borwein, and M.~K. Tam}, {\em
  Douglas--{R}achford feasibility methods for matrix completion problems},
  ANZIAM J., 55 (2014), pp.~299--326.

\bibitem{abtam1}
\leavevmode\vrule height 2pt depth -1.6pt width 23pt, {\em Recent results on
  {D}ouglas--{R}achford methods for combinatorial optimization problems}, J.
  Optim. Theory App., 163 (2014), pp.~1--30.

\bibitem{aacampoy}
{\sc F.~J. Arag{\'o}n~Artacho, R.~Campoy, and M.~K. Tam}, {\em The
  {D}ouglas--{R}achford algorithm for convex and nonconvex feasibility
  problems}, Math. Method Oper. Res.,  (2019), pp.~1--40.

\bibitem{bbauschke}
{\sc H.~H. Bauschke and J.~M. Borwein}, {\em On projections algorithms for
  solving convex feasibility problems}, SIAM Rev., 38 (1996), pp.~367--426.

\bibitem{bauschke1997method}
{\sc H.~H. Bauschke, J.~M. Borwein, and A.~S. Lewis}, {\em The method of cyclic
  projections for closed convex sets in {H}ilbert space}, Contemp. Math., 204
  (1997), pp.~1--38.

\bibitem{bcombettes}
{\sc H.~H. Bauschke and P.~L. Combettes}, {\em Convex Analysis and Monotone
  Operator Theory in Hilbert Spaces}, Springer, Cham, 2017.

\bibitem{bdao}
{\sc H.~H. Bauschke and M.~N. Dao}, {\em On the finite convergence of the
  {D}ouglas--{R}achford algorithm for solving (not necessarily convex)
  feasibility problems in {E}uclidean spaces}, SIAM J. Optim., 27 (2017),
  pp.~507--537.

\bibitem{bnphan}
{\sc H.~H. Bauschke, D.~Noll, and H.~M. Phan}, {\em Linear and strong
  convergence of algorithms involving averaged nonexpansive operators}, J.
  Math. Anal. Appl., 421 (2015), pp.~1--20.

\bibitem{bsims}
{\sc J.~M. Borwein and B.~Sims}, {\em The {D}ouglas--{R}achford algorithm in
  the absence of convexity}, in Fixed-point Algorithms for Inverse Problems in
  Science and Engineering, Springer, New York, 2011, pp.~93--109.

\bibitem{btam}
{\sc J.~M. Borwein and M.~K. Tam}, {\em A cyclic {D}ouglas--{R}achford
  iteration scheme}, J. Optim. Theory Appl., 160 (2014), pp.~1--29.

\bibitem{btam2}
\leavevmode\vrule height 2pt depth -1.6pt width 23pt, {\em The cyclic
  {D}ouglas--{R}achford feasibility method for inconsistent feasibility
  problems}, J. Nonlinear Convex A., 16 (2015), pp.~537--584.

\bibitem{btam1}
\leavevmode\vrule height 2pt depth -1.6pt width 23pt, {\em Reflection methods
  for inverse problems with applications to protein conformation
  determination}, in Generalized Nash Equilibrium Problems, Bilevel Programming
  and MPEC, Springer, 2017, pp.~83--100.

\bibitem{bregman}
{\sc L.~Bregman}, {\em The method of successive projection for finding a common
  point of convex sets}, Sov. Math. Dok., 6 (1965), pp.~688--692.

\bibitem{cheney}
{\sc W.~Cheney and A.~A. Goldstein}, {\em Proximity maps for convex sets}, P.
  Am. Math. Soc., 10 (1959), pp.~448--450.

\bibitem{dphan1}
{\sc M.~N. Dao and H.~M. Phan}, {\em Linear convergence of the generalized
  {D}ouglas--{R}achford algorithm for feasibility problems}, J. Global Optim.,
  72 (2018), pp.~443--474.

\bibitem{dphan2}
\leavevmode\vrule height 2pt depth -1.6pt width 23pt, {\em Linear convergence
  of projection algorithms}, Math. Oper. Res., 44 (2019), pp.~715--738.

\bibitem{dtam}
{\sc M.~N. Dao and M.~K. Tam}, {\em A {L}yapunov-type approach to convergence
  of the {D}ouglas--{R}achford algorithm for a nonconvex setting}, J. Global
  Optim., 73 (2019), pp.~83--112.

\bibitem{daubechies}
{\sc I.~Daubechies}, {\em Ten Lectures on Wavelets}, Society for Industrial and
  Applied Mathematics, Philadelphia, Pennsylvania, 1992.

\bibitem{deutsch}
{\sc F.~Deutsch}, {\em Best Approximation in Inner Product Spaces},
  Springer-Verlag, New York, USA, 2001.

\bibitem{dhlakey}
{\sc N.~D. Dizon, J.~A. Hogan, and J.~D. Lakey}, {\em Optimization in the
  construction of nearly cardinal and nearly symmetric wavelets}, in 13th
  International conference on Sampling Theory and Applications (SampTA), IEEE,
  2019, pp.~1--4.

\bibitem{drachford}
{\sc J.~Douglas and H.~Rachford}, {\em On the numerical solution of heat
  conduction problems in two and three space variables}, T. A. Math. Soc., 82
  (1956), pp.~421--439.

\bibitem{franklin}
{\sc D.~J. Franklin}, {\em Projective Algorithms for Non-separable Wavelets and
  Clifford Fourier Analysis}, PhD thesis, The University of Newcastle
  (Australia), 2018.

\bibitem{fhtam}
{\sc D.~J. Franklin, J.~A. Hogan, and M.~K. Tam}, {\em Higher-dimensional
  wavelets and the {D}ouglas-{R}achford algorithm}, in 13th International
  Conference on Sampling Theory and Applications (SampTA), IEEE, 2019,
  pp.~1--4.

\bibitem{fhtamfull}
\leavevmode\vrule height 2pt depth -1.6pt width 23pt, {\em A
  {D}ouglas--{R}achford construction of non-separable continuous compactly
  supported multidimensional wavelets}, arXiv preprint arXiv:2006.03302,
  (2020).

\bibitem{halperin}
{\sc I.~Halperin}, {\em The product of projection operators}, Acta. Sci. Math.
  (Szeged), 23 (1962), pp.~96--99.

\bibitem{kruger}
{\sc A.~Y. Kruger}, {\em About regularity of collections of sets}, Set-Valued
  Anal., 14 (2006), pp.~187--206.

\bibitem{llmalick}
{\sc A.~S. Lewis, D.~R. Luke, and J.~Malick}, {\em Local linear convergence for
  alternating and averaged nonconvex projections}, Found. Comput. Math., 9
  (2009), pp.~485--513.

\bibitem{lions}
{\sc P.~Lions and B.~Merceir}, {\em Splitting algorithms for the sum of two
  nonlinear operators}, SIAM J. Numer. Anal., 16 (1979), pp.~964--979.

\bibitem{russell2018quantitative}
{\sc D.~R. Luke, N.~H. Thao, and M.~K. Tam}, {\em Quantitative convergence
  analysis of iterated expansive, set-valued mappings}, Math. Oper. Res., 43
  (2018), pp.~1143--1176.

\bibitem{mallat}
{\sc S.~Mallat}, {\em Multiresolution approximations and wavelet orthonormal
  bases of $ {L}_2(\mathbb{R})$}, T. A. Math. Soc., 315 (1989), pp.~69--87.

\bibitem{meyer}
{\sc Y.~Meyer}, {\em Wavelets and Operators}, Cambridge University Press,
  Cambridge, UK, 1993.

\bibitem{mordukhovich}
{\sc B.~S. Mordukhovich}, {\em Variational Analysis and Generalized
  Differentiation I}, Springer, Berlin, 2006.

\bibitem{opial}
{\sc Z.~Opial}, {\em Weak convergence of the sequence of successive
  approximations for nonexpansive mappings}, B. Am. Math. Soc., 73 (1967),
  pp.~591--597.

\bibitem{phan}
{\sc H.~M. Phan}, {\em Linear convergence of the {D}ouglas--{R}achford method
  for two closed sets}, Optimization, 65 (2016), pp.~369--385.

\bibitem{pierra}
{\sc G.~Pierra}, {\em Decomposition through formalization in a product space},
  Math. Program., 28 (1984), pp.~96--115.

\bibitem{svaiter}
{\sc B.~F. Svaiter}, {\em On weak convergence of the {D}ouglas--{R}achford
  method}, SIAM J. Control Optim., 49 (2011), pp.~280--287.

\bibitem{neumann}
{\sc J.~von Neumann}, {\em Functional Operators Volume II: The Geometry of
  Orthogonal Spaces}, Princeton University Press, New Jersey, USA, 1950.

\bibitem{xia}
{\sc X.~Xia and Z.~Zhang}, {\em On sampling theorem, wavelets, and wavelet
  transforms}, IEEE T. Signal Proces., 41 (1993), pp.~3524--3535.

\end{thebibliography}

\end{document}